\documentclass[12pt]{amsart}
\usepackage[dvipsnames]{xcolor}
\usepackage{amsfonts,amsbsy, amsthm, amsmath, amssymb, bbm, bm, color,  enumerate,tikz-cd,latexsym,amsopn,amstext, amsxtra,euscript,amscd,cite,hyperref}
\usepackage{enumitem}
\usepackage{soul}

\usepackage[all]{xy}
\usepackage[margin=2.6cm]{geometry}
\usepackage{float}
\usepackage{mathrsfs, mathtools}

\usepackage{helvet}

\numberwithin{equation}{section}

\allowdisplaybreaks

\newtheorem{theorem}{Theorem}
\numberwithin{theorem}{section}
\newtheorem{problem}[theorem]{Problem}
\newtheorem{lemma}[theorem]{Lemma}

\newtheorem{conjecture}[theorem]{Conjecture}
\newtheorem{remark}[theorem]{Remark}

\theoremstyle{definition}

\newcommand{\R}{\mathbb{R}}
\newcommand{\Q}{\mathbb{Q}}

\newcommand{\dd}{\mathrm{d}}
\newcommand{\N}{\mathrm{N}}
\newcommand{\Res}{\mathrm{Res}}
\newcommand{\coker}{\mathrm{coker}}
\newcommand{\Cl}{\mathrm{Cl}}

\newcommand{\kn}{\mathfrak{n}}

\newcommand{\kp}{\mathfrak{p}}
\newcommand{\kq}{\mathfrak{q}}
\newcommand{\km}{\mathfrak{m}}
\newcommand{\kd}{\mathfrak{d}}

\renewcommand{\Re}{\mathop{\mathrm{Re}}}

\renewcommand{\tilde}{\widetilde}

\begin{document}
\title{On Hypothesis H of Rudnick and Sarnak}

\author{Yujiao Jiang}

\address{Yujiao Jiang\\
    School of Mathematics and Statistics $\&$ State Key Laboratory of Cryptography and Digital Economy Security
    \\
    Shandong University
    \\
    Weihai
    \\
    Shandong 264209
    \\
    China}
\email{yujiaoj@sdu.edu.cn}

\date{\today}

\begin{abstract}
We prove Hypothesis H in full generality for ${\rm GL}_n$ over any number field. This result is a consequence of our stronger effective bound on Euler products involving Rankin--Selberg coefficients at prime ideal powers. The proof rests on a new analytic method, which employs a power sieve over number fields and an iterative argument to bypass the functoriality barrier that had restricted prior results to $n\leq 4$. As applications, we unconditionally establish smoothed $m$-level correlations of zeros of automorphic $L$-functions for a restricted class of test functions, provide an explicit polynomial bound for the strong multiplicity one problem for coefficients, and resolve the Selberg orthogonality conjecture with stronger error terms.
\end{abstract}
	
\maketitle

\section{Introduction} 

The generalized Ramanujan conjecture (GRC) is one of the foundational problems in modern number theory. It predicts optimal bounds for the local components of automorphic representations and has far-reaching implications for the analytic theory of automorphic forms and $L$-functions. Although the GRC remains out of reach in general, significant progress has been made in special cases, and various approximations to the GRC have already had substantial impact on analytic number theory. In many situations, these approximations are sufficient to derive deep arithmetic consequences. For further discussion, see e.g. \cite{BB-2013, Sarnak-2005}.

To state the conjecture precisely, we introduce some notation. Let $F$ be a number field with ring of integers $\mathcal{O}_F$, absolute norm $\N=\N_{F/\Q}$, and absolute discriminant $D_F$. Let $\mathbb{A}_{F}$ denote its ring of adeles. For an integer $n\geq 1$, let $\mathfrak{F}_n$ be the set of cuspidal automorphic representations $\pi$ of $\mathrm{GL}_n(\mathbb{A}_{F})$ whose central character $\omega_{\pi}$ is unitary. For each $\pi\in\mathfrak{F}_n$, the associated standard $L$-function is defined by the Euler product 
\[
L(s,\pi)=\prod_{\kp} \prod_{j=1}^{n}\Big(1-\frac{\alpha_{j,\pi}(\kp)}{\N\kp^{s}}\Big)^{-1}
\]
for $\Re s>1$, where the $\alpha_{j,\pi}(\kp)$ are suitable complex numbers. The GRC predicts that for every prime ideal $\kp$ at which $\pi_{\kp}$ is unramified, the Satake parameters satisfy 
\[
|\alpha_{j,\pi}(\kp)| = 1 \quad\text{for all }\, j = 1, \ldots, n.
\]
This conjecture is only known in certain cases. For instance, when $n = 1$, the representation $\pi$ is a unitary Hecke character, and then the conjecture holds unconditionally. Another well-known case is when $F=\Q$ and $n = 2$, in which $\pi$ corresponds to a classical holomorphic cusp form of even integral weight. In this setting, the conjecture follows from Deligne's proof of the Weil conjecture \cite{Deligne-1974}. However, the general case remains an open problem. The full strength of the GRC is often unnecessary for a wide range of analytic applications. Weaker forms, particularly those controlling individual or average behavior of the Satake parameters, are often sufficient. A notable example of such an approximation, introduced by Rudnick and Sarnak \cite{RS-1996}, is the following hypothesis concerning second moments of coefficients associated with $L(s,\pi)$ at prime powers. While their original formulation was over $\Q$, it can easily be extended to any number field $F$.

\vskip 2mm
\noindent \textbf{Hypothesis H}. For any fixed $k \geq 2$ and any fixed $\pi\in \mathfrak{F}_n$,
\[
\sum_{\kp} \frac{(\log \N\kp)^2|a_\pi(\kp^k)|^2}{\N\kp^{k}}<\infty,
\]
where $a_\pi(\kp^k)=\sum_{j=1}^n \alpha_{j,\pi}(\kp)^{k}$.
\begin{remark}
	~
\begin{enumerate}[label=\arabic*.]
\item Hypothesis H is a direct consequence of the GRC.
\item Hypothesis H is only known to hold for automorphic $L$-functions of small degree. While the initial verifications were carried out over $F=\Q$, the arguments used can be adapted without difficulty to prove the same results for any general number field $F$. For instance,
\begin{itemize} 
    \item For $n=2$ and $n=3$, the hypothesis follows from the Rankin--Selberg theory.
    \item For $n=4$, it was established by Kim \cite{Kim-2006} through the functoriality of the exterior square lift $\wedge^2 \pi$ of a cuspidal representation $\pi$ of $\mathrm{GL}_4(\mathbb{A}_{\mathbb{Q}})$ (see \cite{Kim-2003}).
    \item Hypothesis H has also been verified in specific cases in $\mathrm{GL}_5(\mathbb{A}_{\mathbb{Q}})$ and $\mathrm{GL}_6(\mathbb{A}_{\mathbb{Q}})$, such as the symmetric fourth power $\operatorname{sym}^4 \pi$ of a cuspidal representation $\pi$ of $\mathrm{GL}_2(\mathbb{A}_{\mathbb{Q}})$, which is an automorphic representation of $\mathrm{GL}_5(\mathbb{A}_{\mathbb{Q}})$, and automorphic representations $\pi_2 \boxtimes \pi_3$ and $\wedge^2 \pi_4$ of $\mathrm{GL}_6(\mathbb{A}_{\mathbb{Q}})$, where $\pi_d$ is a cuspidal representation of $\mathrm{GL}_d(\mathbb{A}_{\mathbb{Q}})$. These two cases are due to the works of Kim \cite{Kim-2006} and Wu--Ye \cite{WY-2007}, respectively.
\end{itemize}
\end{enumerate}
\end{remark}

Hypothesis H, though technical in its formulation, has emerged as a powerful tool in the analytic study of automorphic $L$-functions. It provides effective control over the coefficients associated with these functions at higher prime powers. In particular, Hypothesis H is used to bound the contribution from prime powers in various summations, ensuring convergence and enabling precise asymptotics in many problems.

A central reason for its significance is that Hypothesis H can often serve as a substitute for the GRC in analytic applications. The second moment estimates that it provides are frequently sufficient to obtain unconditional results on deep problems related to automorphic forms and their $L$-functions.

One of the most influential applications of Hypothesis H appears in the work of Rudnick and Sarnak \cite{RS-1996}, where they studied the $n$-level correlations of the zeros of automorphic $L$-functions. Under the assumption of Hypothesis H, they established that the statistical distribution of these zeros agrees with the predictions for the Gaussian unitary ensemble from random matrix theory. 

Moreover, Hypothesis H yields arithmetic estimates that contribute to progress on several deep arithmetic problems, such as the Selberg orthogonality conjecture in the automorphic context. In this direction, it has been a key assumption in the work of Liu and Ye on orthogonality relations \cite{LY-2005,LY-2006}, in moment estimates of automorphic $L$-functions by Milinovich and Turnage-Butterbaugh \cite{MT-2014}, and in the study of extreme central values of $L$-functions in families by Heap and Li \cite{HL-2024}.

Through these and other applications, Hypothesis H has proven to be a flexible and effective analytic input in the study of automorphic forms. It allows one to bypass the GRC while still achieving unconditional results in problems involving moments, orthogonality, and value distributions of automorphic $L$-functions.

In this paper, we establish Hypothesis H in full generality for all $\pi \in \mathfrak{F}_n$, thereby removing this technical hypothesis from a wide range of applications.

\begin{theorem}\label{thm-HH}
	Hypothesis H is true for all $n\geq 1$.
\end{theorem}

However, in some applications, Hypothesis H does not suffice. For example, in the study of log-free zero density estimates for automorphic $L$-functions over $F=\Q$, Kowalski and Michel \cite{KM-2002} introduced a stronger hypothesis, namely that $|\alpha_{j,\pi}(\kp)| \leq \N\kp^{1/4 -\delta}$ for some $\delta > 0$. Subsequently, also working over the rational field, Brumley, Thorner, and Zaman \cite{BTZ-2022} proposed a more flexible hypothesis involving the infinite product
\begin{equation}\label{eq-hypotheis-BTZ}
    \prod_{\kp} \sum_{k=0}^\infty \frac{\max_{1\leq j\leq n} |\alpha_{j,\pi}(\kp)|^{2k}}{\N\kp^{k(1+\varepsilon)}} \ll C(\pi)^{\varepsilon},
\end{equation}
which holds for $n \leq 4$ when $F=\Q$ and was extended to general number fields in \cite[Lemma 7.1]{HT-2022}. Here $C(\pi)$ is the analytic conductor of $\pi$ (see \eqref{eq-conductor} for the definition). Using a straightforward estimation together with the Satake parameter bound \eqref{eq-LRS-finite} due to Luo--Rudnick--Sarnak, one sees that this condition is in fact equivalent to
\[
\prod_{\kp} \left(1 + \frac{\max_{1\leq j\leq n} |\alpha_{j,\pi}(\kp)|^2}{\N\kp^{1+\varepsilon}} \right) \ll C(\pi)^\varepsilon.
\]
The above hypothesis plays a central role in their work on large sieve inequalities for $\mathrm{GL}_n$ and in establishing log-free zero density estimates for families of Rankin--Selberg $L$-functions. It was also used by Humphries and Thorner \cite{HT-2022} to obtain the prime number theorem in short intervals for automorphic forms on $\mathrm{GL}_n$. In these applications, a key technical requirement is the ability to make the dependence on the conductor $C(\pi)$ explicit and uniform across families, which is a feature not addressed by Hypothesis H.

In this paper, we establish a new bound that not only implies Hypothesis H but also suffices for the aforementioned applications where Hypothesis H is inadequate. Our result provides refined control over certain infinite products involving coefficients of Rankin--Selberg $L$-functions, with an explicit and uniform dependence on the conductor. To state our main theorem precisely, we now introduce the necessary notation. Let $\tilde{\pi}$ denote the contragredient of $\pi \in \mathfrak{F}_n$, which also belongs to $\mathfrak{F}_n$. Let $L(s, \pi \times \tilde{\pi})$ be the associated Rankin--Selberg $L$-function. We write
\[
L(s, \pi \times \tilde{\pi}) = \sum_{\kn} \frac{\lambda_{\pi \times \tilde{\pi}}(\kn)}{\N\kn^{s}}, \quad 
- \frac{L'}{L}(s, \pi \times \tilde{\pi}) = \sum_{\kn} \frac{\Lambda_F(\kn) a_{\pi \times \tilde{\pi}}(\kn)}{\N\kn^{s}},
\]
where $\kn$ runs over the nonzero ideals of $\mathcal{O}_F$, and $\lambda_{\pi \times \tilde{\pi}}(\kn)$ and $\Lambda_F(\kn) a_{\pi \times \tilde{\pi}}(\kn)$ denote the Dirichlet coefficients of the Rankin--Selberg $L$-function and its logarithmic derivative, respectively.

\begin{theorem}\label{thm-main}
For any $\varepsilon>0$ and any $\pi \in \mathfrak{F}_n$, we have
\[
\prod_{\kp}\left(1+\sum_{k=2}^{\infty} \frac{\lambda_{\pi\times\tilde{\pi}}(\kp^k)}{\N\kp^{k\sigma}}\right)\ll C(\pi)^{\varepsilon},
\]
and 
\[
\prod_{\kp}\left(1+\sum_{k=2}^{\infty} \frac{a_{\pi\times\tilde{\pi}}(\kp^k)}{\N\kp^{k\sigma}}\right)\ll C(\pi)^{\varepsilon}
\]
for any $\sigma \geq 1-\frac{1}{n^2+1}+\varepsilon$, where $C(\pi)$ is the analytic conductor of $\pi$, and the implied constants depend only on $[F:\Q],n$ and $\varepsilon$.
\end{theorem}

\begin{remark}
Compared to Hypothesis H which is expressed in terms of the arithmetic function $|a_\pi(\kn)|^2$, our theorem is formulated using the Rankin--Selberg coefficients $\lambda_{\pi \times \tilde{\pi}}(\kn)$, and in fact also covers the coefficients  $a_{\pi \times \tilde{\pi}}(\kn)$ of the logarithmic derivative. This formulation proves more flexible in analytic applications, since many arithmetic quantities of interest can be controlled using $\lambda_{\pi \times \tilde{\pi}}(\kn)$ or $a_{\pi \times \tilde{\pi}}(\kn)$. See the inequalities in Section 3 for some examples.

However, we emphasize that our bound does not imply the stronger hypothesis \eqref{eq-hypotheis-BTZ} introduced by Brumley, Thorner, and Zaman, which involves a refined infinite sum over the Satake parameters and is currently known only for $n \leq 4$. 
\end{remark}

\textbf{Proof of Theorem~\ref{thm-HH}.}
As a consequence, we now deduce Hypothesis H as a corollary of our main result. Recall the elementary inequality
\begin{equation}\label{eq-ineq-sumprod}
    \sum_{\kp} x_{\kp}\leq \prod_{\kp}(1+x_{\kp}),
\end{equation}
which holds for any sequence $\{x_{\kp}\}_{\kp}$ of non-negative real numbers. The second assertion in Theorem~\ref{thm-main}, combined with the inequality \eqref{eq-ineq-sumprod} with $x_{\kp}=\sum_{k=2}^{\infty} \frac{a_{\pi\times\tilde{\pi}}(\kp^k)}{\N\kp^{k\sigma}}$, yields
\[
\sum_{\kp}\sum_{k=2}^{\infty} \frac{a_{\pi\times\tilde{\pi}}(\kp^k)}{\N\kp^{k\sigma}}\ll C(\pi)^{\varepsilon}
\]
for any $\sigma \geq 1-\frac{1}{n^2+1}+\varepsilon$. Theorem~\ref{thm-HH} follows immediately from this and the inequality~\eqref{coeffpair}.

\vskip 2mm
\subsection{Strategy of the proof and comparison with previous methods}

The proof of our main theorem rests upon a new analytic framework that differs fundamentally from prior work \cite{RS-1996,Kim-2006,WY-2007} on Hypothesis H. To motivate our approach, we begin by outlining the previous strategy and pointing out the main difficulty that prevented further progress.

The previous approach to establishing Hypothesis H relied on obtaining pointwise control over the maximal Satake parameter, $\max_{1\leq j\leq n}|\alpha_{j,\pi}(\kp)|$, through an analysis of its parameterization. Such control yields the bound $|a_{\pi}(\kp^k)|^2\leq n^2\max_{1\leq j\leq n}|\alpha_{j,\pi}(\kp)|^{2k}$ for any $k\geq 2$. In the case $n=3$, the unitarity condition uniquely constrains this parameterization, which leads to the estimate
\[
\max_{1\leq j\leq 3}|\alpha_{j,\pi}(\kp)|^2 \ll 1+|a_\pi(\kp)|^{2}.
\]
This inequality provides the necessary pointwise control. The proof is then completed using the Rankin--Selberg theory.

The same strategy for parameterizing the Satake parameters extends to $n=4$, but the analysis becomes more complex. The parameterization leads to several distinct cases. For the more intricate of these, the standard bounds towards the GRC for $\mathrm{GL}_4$ are insufficient. The key insight of Kim \cite{Kim-2006} was to employ the functoriality of the exterior square lift $\Lambda^2: \mathrm{GL}_4 \to \mathrm{GL}_6$. By applying the bound \eqref{eq-LRS-finite} for $\mathrm{GL}_6$ (due to Luo, Rudnick and Sarnak) to this lift, he established the estimate
\[
\max_{1\leq j\leq 4}|\alpha_{j,\pi}(\kp)|^2 \ll \N \kp^{\frac{1}{2}-\frac{1}{37}}+|a_\pi(\kp)|^{2}.
\]
This estimate is sufficient to handle all cases and thus completes the proof of Hypothesis H for $n=4$.

The principle of using lifts to achieve pointwise control was developed by Brumley \cite{Brumley-2006-arch}, who proved the estimate
\[
\max_{1\leq j\leq n}|\alpha_{j,\pi}(\kp)|^2 \ll_n 1+|a_\pi(\kp)|^{2}+\sum_{\nu=2}^{\lfloor n/2 \rfloor} |a_{\Lambda^{\nu}\pi}(\kp)|^{\frac{2}{\nu}},
\]
where $\Lambda^{\nu}\pi$ denotes the $\nu$-th exterior power lift of $\pi$. This inequality makes the obstruction explicit. To bound the maximal Satake parameter on the left-hand side, one must control the coefficients of the higher exterior power lifts on the right-hand side. For $n \geq 5$, however, this requires the automorphy of these lifts, which is an unproven part of the Langlands program. This functoriality barrier has prevented any significant progress for nearly two decades.

In contrast, our approach bypasses the strategy of pointwise control. We introduce a global and analytic method that directly estimates the average of the Rankin--Selberg coefficients $\lambda_{\pi\times\tilde{\pi}}(\kn^k)$  with $k\geq 2$. The method proceeds in three steps. First, the power sieve reduces the problem to estimating sums twisted by characters. We establish an initial estimate for these sums by applying the convexity bound for the twisted Rankin--Selberg $L(s, \pi \times \tilde{\pi} \otimes \psi)$.  Second, we employ a robust iterative procedure inspired by Iwaniec \cite{Iwaniec-1990}. This procedure systematically refines the initial estimate, reducing the exponent of the dependence on the conductor to be arbitrarily small. Finally, the analysis is completed by combining our new bound with the pointwise estimate of Luo, Rudnick and Sarnak. Our bound is applied for small exponents ($k\leq n^2+1$), while their bound is used for larger exponents.

The analytic theory of the twisted $L$-function $L(s, \pi \times \tilde{\pi} \otimes \psi)$ plays an important role in our method, as it did in the seminal work \cite{LRS-1995,LRS-1999} of Luo, Rudnick and Sarnak on the GRC. However, the two approaches diverge sharply in their application of this theory. Luo, Rudnick and Sarnak used information about the distribution of zeros of this $L$-function to constrain individual parameters. This led to their pointwise bound \eqref{eq-LRS-finite}. In contrast, we use properties of this $L$-function to estimate the character sums produced by the sieve. This estimate then serves as the input for an iterative process, allowing us to control averages rather than individual values.

In essence, our work replaces a dependence on deep conjectural results from the Langlands program with a new framework of analytic machinery.

\vskip 5mm
\textbf{Notation.} Throughout the paper, $\varepsilon$ represents an arbitrarily small, positive constant, which may vary from one occurrence to the next. We use $f = O(g), f \ll g, \text{ or } g \gg f$ to denote the bound $|f| \leq C g$ for some positive constant $C$. If this constant $C$ depends on additional parameters then we will indicate this by subscripts. For instance, $f \ll_{\delta,\varepsilon} g$ denotes the bound $|f| \leq C_{\delta,\varepsilon} \,g$ for some $C_{\delta,\varepsilon}$ depending on $\delta, \varepsilon$. In some cases, we omit explicit dependence in the constants, which may depend at most on $n$, $[F:\Q]$, and $\varepsilon$.

\section{Properties of \texorpdfstring{$L$}{L}-functions}

In this section, we give basic notation and definitions of standard and Rankin--Selberg $L$-functions, and their twists, including their fundamental analytic properties and functional equations. See \cite[Section 1]{Brumley-2006-ajm}, \cite[Lecture 1]{Michel-2007} and \cite[Section 2]{ST-2019} for detailed discussions.

\subsection{Standard $L$-functions}

Let $\pi=\otimes_{v} \pi_v \in \mathfrak{F}_n$ be a cuspidal automorphic representation of $\mathrm{GL}_n(\mathbb{A}_{F})$, where $v$ varies over the places of a number field $F$. Denote by $\kq_{\pi}$ the arithmetic conductor of $\pi$. For each nonarchimedean place $v$, there corresponds a prime ideal $\kp$ of $\mathcal{O}_F$. At each prime ideal $\kp$ such that $\pi_{\kp}$ is unramified, the inverse of the local factor $L(s,\pi_{\kp})$ is defined by a polynomial in $\N\kp^{-s}$ of degree $n$
\[
L(s,\pi_{\kp})^{-1}=\prod_{j=1}^{n}\Big(1-\frac{\alpha_{j,\pi}(\kp)}{\N\kp^{s}}\Big),
\]
where ${\alpha_{1,\pi}(\kp), \ldots, \alpha_{n,\pi}(\kp)}$ are the Satake parameters of $\pi$ at $\kp$, all of which are nonzero complex numbers. At $\kp$ such that $\pi_{\kp}$ is ramified, the local factor $L(s, \pi_{\kp})$ can still be expressed in the same form, with some of the parameters $\alpha_{j,\pi}(\kp)$ equal to zero. The standard $L$-function $L(s,\pi)$ associated to $\pi$ is of the form
\[
L(s,\pi)=\prod_{\kp} L(s,\pi_{\kp})=\sum_{\kn}\frac{\lambda_{\pi}(\kn)}{\N\kn^{s}},
\]
where $\kn$ runs over all nonzero ideals of $\mathcal{O}_F$, and $\kp$ runs over all nonzero prime ideals of $\mathcal{O}_F$. The Euler product and Dirichlet series converge absolutely for $\Re(s)>1$ by the Rankin--Selberg theory. 

At each archimedean place, there are $n$ complex Langlands parameters $\{\mu_{1, \pi}(v),\ldots,\mu_{n, \pi}(v)\}$ from which we define
\[
L(s,\pi_{\infty}) = \prod_{ v|\infty}\prod_{j=1}^{n}\Gamma_{v}(s+\mu_{j,\pi}(v)),\qquad \Gamma_{v}(s):=\begin{cases}
	\pi^{-s/2}\Gamma(s/2)&\mbox{if $F_{ v}=\R$,}\\
	2(2\pi)^{-s}\Gamma(s)&\mbox{if $F_{ v}=\mathbb{C}$.}
\end{cases}
\]
It is known that there exists a constant
\begin{equation}\label{eq-delta-n}
	\delta_{n} \in \left[0,\frac{1}{2}-\frac{1}{n^2+1}\right]
\end{equation}
such that for all places $\kp,v$ and indices $j$,
\begin{equation}\label{eq-LRS-finite}
	|\alpha_{j,\pi}(\kp)|\leq  \N\kp^{\delta_{n}}\qquad\textup{and}\qquad\Re(\mu_{j, \pi}(v))\geq -\delta_{n}.
\end{equation}
The bounds follow from the works of Luo--Rudnick--Sarnak \cite{LRS-1999} and M\"uller--Speh \cite{MS-2004}. The generalized Ramanujan and Selberg conjectures predict that one may take $\delta_{n}=0$ in \eqref{eq-delta-n}.

Let $\tilde{\pi}$ denote the contragredient of $\pi\in\mathfrak{F}_n$, which is also an irreducible cuspidal automorphic representation in $\mathfrak{F}_n$. For any place $v$ of $F$, $\tilde{\pi}_{v}$ is equivalent to the complex conjugate $\bar{\pi}_{v}$, and hence we have
\[
\big\{\alpha_{j,\tilde{\pi}}(\kp):\ 1\leq j\leq n \big\}=\big\{\overline{ \alpha_{j,\pi}(\kp)}:\ 1\leq j\leq n \big\}
\]
and
\[
\big\{\mu_{j,\tilde{\pi}}(v):\ 1\leq j\leq n \big\}=\big\{\overline{\mu_{j, \pi}(v)}:\ 1\leq j\leq n \big\}.
\]
Let $d(v)=1$ if $F_{v}=\R$ and $d(v)=2$ if $F_{v}=\mathbb{C}$. The analytic conductor of $\pi$ is defined as
\begin{equation}\label{eq-conductor}
C(\pi,t)=D_F^n \N\kq_{\pi}\prod_{ v|\infty}\prod_{j=1}^n(3+|it+\mu_{j,\pi}( v)|^{d( v)}),\qquad C(\pi)=C(\pi,0).
\end{equation}

Define the completed $L$-function by
\[
\Lambda(s,\pi) = (D_F^n \N\kq_{\pi})^{s/2}L(s,\pi_{\infty})L(s,\pi).
\]
Then $\Lambda(s,\pi)$ has an analytic continuation to the entire complex plane, is bounded in vertical strips, and satisfies the functional equation
\[
\Lambda(s,\pi)=W(\pi)\Lambda(1-s,\tilde{\pi}),
\]
where $W(\pi)$ is a complex number of modulus one. As $\pi$ is cuspidal, this continuation is entire.

\subsection{Rankin--Selberg $L$-functions}
We now turn to the Rankin--Selberg $L$-functions. Let $\pi=\otimes_v \pi_{v}\in\mathfrak{F}_n$ and $\pi'=\otimes_v \pi_{v}'\in\mathfrak{F}_{n'}$. The Rankin--Selberg $L$-function $L(s,\pi\times\pi')$ associated to $\pi$ and $\pi'$ is defined by
\[
L(s,\pi\times\pi')=\prod_{\kp}L(s,\pi_{\kp}\times\pi_{\kp}')=\sum_{\kn}\frac{\lambda_{\pi\times\pi'}(\kn)}{\N\kn^{s}}.
\]
Both the Euler product and the Dirichlet series converge absolutely for $\Re(s)>1$. For each prime ideal $\kp$, the inverse of the local factor $L(s,\pi_{\kp}\times\pi_{\kp}')$ is a polynomial in $\N\kp^{-s}$ of degree at most $nn'$,
\begin{equation*}
L(s,\pi_{\kp}\times\pi_{\kp}')^{-1}=\prod_{j=1}^{n}\prod_{j'=1}^{n'}\Big(1-\frac{\alpha_{j,j',\pi\times\pi'}(\kp) }{\N\kp^{s}}\Big),
\end{equation*}
for suitable complex numbers $\alpha_{j,j',\pi\times\pi'}(\kp)$.  With $\delta_{n}$ as in \eqref{eq-delta-n}, we have the pointwise bound
\begin{equation}\label{eq-LRS-2}
	|\alpha_{j,j',\pi\times\pi'}(\kp)|\leq  \N\kp^{\delta_{n}+\delta_{n'}}.
\end{equation}
If $\kp\nmid \kq_{\pi}\kq_{\pi'}$ (i.e., if both $\pi$ and $\pi'$ are unramified at $\kp$), then the local parameters satisfy the equality of multisets
\begin{equation}\label{eq-RS-unramified}
	\big\{\alpha_{j,j',\pi\times\pi'}(\kp):\  j\leq n,\,  j'\leq n' \big\}
	=\big\{ \alpha_{j,\pi}(\kp)\alpha_{j',\pi'}(\kp):\  j\leq n,\,  j'\leq n'\big\}.
\end{equation}
See Brumley \cite[Appendix]{ST-2019} for a description of $\alpha_{j,j',\pi\times\pi'}(\kp)$ when $\kp|\kq_{\pi}\kq_{\pi'}$. At each archimedean place $v$ of $F$, there exist $nn'$ complex Langlands parameters $\mu_{j,j',\pi\times\pi'}(v)$ associated to $\pi_{v}$ and $\pi_{v}'$, from which one defines
\[
L(s,\pi_{\infty}\times\pi_{\infty}') = \prod_{ v|\infty}\prod_{j=1}^{n}\prod_{j'=1}^{n'}\Gamma_{v}(s+\mu_{j,j',\pi\times\pi'}(v)).
\]
These parameters satisfy the pointwise bound
\begin{equation*}
	\Re(\mu_{j,j',\pi\times\pi'})\geq-\delta_{n}-\delta_{n'}.
\end{equation*}

Let $\kq_{\pi\times\pi'}$ be the arithmetic conductor of $\pi\times\pi'$. The completed $L$-function
\[
\Lambda(s,\pi\times\pi')=(D_F^{n'n}\N\kq_{\pi\times\pi'})^{s/2}L(s,\pi\times\pi')L(s,\pi_{\infty}\times\pi_{\infty}')
\]
has a meromorphic continuation to $\mathbb{C}$ and is bounded in vertical strips. It is entire unless $\pi'\cong \tilde{\pi}\otimes |\cdot|^{i\tau}$ for some $\tau \in \R$, in which case it has simple poles at $s=-i\tau, 1-i\tau$. Furthermore, it satisfies the functional equation
\begin{equation*}
\Lambda(s,\pi\times\pi')=W(\pi\times\pi')\Lambda(1-s,\tilde{\pi}\times\tilde{\pi}'),
\end{equation*}
where $W(\pi\times\pi')$ is a complex number of modulus one. 

As with $L(s,\pi)$, we define the analytic conductor of $\pi\times\pi'$ to be
\[
C(\pi\times\pi',t):=D_F^{n'n}\N\kq_{\pi\times\pi'}\prod_{ v|\infty}\prod_{j=1}^n \prod_{j'=1}^{n'}(3+|it+\mu_{j,j',\pi\times\pi'}( v)|^{d( v)}),\qquad C(\pi\times\pi')=C(\pi\times\pi',0).
\]
To decouple the dependencies of $C(\pi\times\pi',t)$ on $\pi$, $\pi'$, and $t$, we cite the result of Bushnell and Henniart \cite[Theorem 1]{BH-1997} and Brumley \cite[Lemma A.2]{HB-2019} which gives
\begin{equation}\label{eq-BH}
	C(\pi\times\pi',t)\leq C(\pi\times\pi')(1+|t|)^{[F:\Q]n'n},\qquad C(\pi\times\pi')\leq e^{O(n'n)} C(\pi)^{n'}C(\pi')^{n}.
\end{equation}

We are especially interested in the case $\pi'\cong \tilde{\pi}$. In this case the Rankin--Selberg $L$-function $L(s,\pi\times\tilde{\pi})$ has non-negative coefficients $\lambda_{\pi\times\tilde{\pi}}(\kn)$. Moreover, $L(s,\pi\times\tilde{\pi})$ extends to the complex plane with a simple pole at $s = 1$.

The convexity bound of the Rankin--Selberg $L$-functions is a crucial tool, which we state as follows.

\begin{lemma}\label{lem-convexity}
Let $\pi\in\mathfrak{F}_n, \pi'\in\mathfrak{F}_{n'}$, and $\sigma\geq 0$. 

\noindent (1) If $\pi' \ncong \tilde{\pi} \otimes |\cdot|^{i\tau}$ for any $\tau \in \mathbb{R}$, then for all $t \in \mathbb{R}$, we have
\begin{equation}\label{eq-convexity-1}
    L\Big(\sigma+it, \pi \times  \pi'\Big)\ll C(\pi \times \pi',t)^{\frac{\max\{1-\sigma,0\}}{2}+\varepsilon}.
\end{equation}

\noindent (2) If $\pi' \cong \tilde{\pi} \otimes |\cdot|^{i\tau}$ for some $\tau \in \mathbb{R}$, then for all $t \in \mathbb{R}$, we have
\begin{equation}\label{eq-convexity-2}
    \bigg( \frac{\sigma+i(t+\tau)-1}{\sigma+i(t+\tau)+1} \bigg) L\Big(\sigma+it, \pi \times  \pi'\Big)\ll C(\pi \times \pi',t)^{\frac{\max\{1-\sigma,0\}}{2}+\varepsilon}.
\end{equation}
The implied constants depend only on $[F:\Q],n,n'$ and $\varepsilon$. 
\end{lemma}

\begin{proof}
For any $t \in \mathbb{R}$, we define the twisted representation $\pi_t \cong \pi \otimes |\cdot|^{it}$, which yields
\[
L(\sigma+it, \pi \times \pi') = L(\sigma, \pi_t \times \pi').
\]
By the definition of the analytic conductor, we also have $C(\pi_t \times \pi') = C(\pi \times \pi', t)$. Thus, it suffices to establish the convexity bound for $L(s, \pi_t \times \pi')$ at the real point $s = \sigma$.

For case (1), the condition $\pi' \ncong \tilde{\pi} \otimes |\cdot|^{i\tau}$ for any $\tau \in \mathbb{R}$ implies that $\pi' \ncong \tilde{\pi}_t \otimes |\cdot|^{i\tau'}$ for any $\tau' \in \mathbb{R}$ (since $\tilde{\pi}_t \cong \tilde{\pi} \otimes |\cdot|^{-it}$). Hence, $L(s, \pi_t \times \pi')$ is entire. Applying \cite[Lemma 3.2]{HT-2025} with $j=0$ directly yields
\[
L(\sigma, \pi_t \times \pi') \ll C(\pi_t \times \pi')^{\frac{\max\{1-\sigma,0\}}{2}+\varepsilon}.
\]
Thus, the estimate \eqref{eq-convexity-1} follows.

For case (2), the condition $\pi' \cong \tilde{\pi} \otimes |\cdot|^{i\tau}$ can be rewritten as $\pi' \cong \tilde{\pi}_t \otimes |\cdot|^{i(t+\tau)}$. This means $L(s, \pi_t \times \pi')$ has a simple pole at $s = 1 - i(t+\tau)$. According to \cite[Lemma 3.2]{HT-2025} evaluated at $s = \sigma$ with $j=0$, we obtain
\[
\bigg( \frac{\sigma+i(t+\tau)-1}{\sigma+i(t+\tau)+1} \bigg) L(\sigma, \pi_t \times \pi') \ll C(\pi_t \times \pi')^{\frac{\max\{1-\sigma,0\}}{2}+\varepsilon},
\]
which further yields the estimate \eqref{eq-convexity-2}.
\end{proof}

\subsection{Twists}  
Let $\pi\in \mathfrak{F}_n$. Let $\km \subsetneq \mathcal{O}_F$ be a nonzero integral ideal of $F$, and let $\psi$ be a primitive ray class character modulo $\km$ with $(\km,\kq_{\pi})=\mathcal{O}_F$. Consider the twisted representation $\pi'=\tilde{\pi}\otimes\psi$.
The associated Rankin--Selberg $L$-function can be written as the Dirichlet series
\[
L(s,\pi\times\tilde{\pi}\otimes\psi)=\sum_{\kn}\frac{\lambda_{\pi\times\tilde{\pi}}(\kn)\psi(\kn)}{\N\kn^{s}}.
\]
The completed $L$-function
\[
\Lambda(s,\pi\times\tilde{\pi}\otimes\psi)= (D_{F}^{n^2} \N\km^{n^2} \N\kq_{\pi\times\tilde{\pi}})^{s/2}L(s,\pi_{\infty}\times\tilde{\pi}_{\infty}\otimes\psi_{\infty}) L(s, \pi\times\tilde{\pi}\otimes\psi)
\]
has an analytic continuation to $\mathbb{C}$ and satisfies the functional equation
\[
\Lambda(s,\pi\times\tilde{\pi}\otimes\psi)=W(\pi\times\tilde{\pi}\otimes\psi)\Lambda(1-s,\tilde{\pi}\times \pi\otimes\overline{\psi}),
\]
where
\[
L(s,\pi_{\infty}\times\tilde{\pi}_{\infty}\otimes\psi_{\infty})=\prod_{ v|\infty}\prod_{j=1}^{n}\prod_{j'=1}^{n}\Gamma_{v}(s+\mu_{j,j',\pi\times\tilde{\pi}\otimes\psi}(v)).
\]
According to the work of M\"uller and Speh \cite[Proof of Lemma 3.1]{MS-2004}, the Langlands parameters $\mu_{j,j',\pi\times\tilde{\pi}\otimes\psi}$ depend only on $\pi$ and the parity of $\psi$.
Applying the general conductor bound \eqref{eq-BH}, we can estimate the analytic conductor of the twisted $L$-function:
\begin{equation*}
C(\pi\times\tilde{\pi}\otimes\psi,t) \ll_{n} C(\pi)^{2n} \N\km^{n^2} (3+|t|)^{[F:\Q]n^2}.
\end{equation*}
Finally, substituting this conductor bound into Lemma \ref{lem-convexity} yields the convexity bound for the twisted $L$-function in the half-plane $\sigma\geq 0$:
\begin{equation}\label{eq-twist-convexity}
L(\sigma+it,\pi\times\tilde{\pi}\otimes\psi) \ll_{[F:\Q],n,\varepsilon} \big(C(\pi)^{2n}\N\km^{n^2} (3+|t|)^{[F:\Q]n^2}\big)^{\frac{\max\{1-\sigma,0\}}{2}+\varepsilon}.
\end{equation}
\vskip 5mm 

\section{Auxiliary inequalities}

In this section, we introduce several arithmetic inequalities that will be used in subsequent arguments. Let $\pi\in\mathfrak{F}_n$ and $\pi'\in\mathfrak{F}_{n'}$. Consider the logarithmic derivative of the Rankin--Selberg $L$-function, given by
\begin{equation}\label{def-api}
	-\frac{L'}{L}(s,\pi\times\pi')=\sum_{\kp} \sum_{\nu=1}^{\infty} \frac{\sum_{j=1}^{n} \sum_{j'=1}^{n'} \alpha_{j, j', \pi \times \pi'}(\kp)^{\nu} \log \N\kp}{\N\kp^{\nu s}}=\sum_{\kn}\frac{\Lambda_F(\kn)a_{\pi\times\pi'}(\kn)}{\N\kn^{s}}
	\end{equation}
	for $\Re(s)>1$, where 
\begin{equation*}
\Lambda_F(\kn):=\begin{cases}
		\displaystyle\log \N\kp,  &\text{ if }\kn=\kp^k \text{ for some }k\in\mathbb{N}, \\
0, &\text{ otherwise,}
\end{cases}
\end{equation*}
and $a_{\pi\times\pi'}(\kp^k)=\sum_{j=1}^{n} \sum_{j'=1}^{n'} \alpha_{j, j', \pi \times \pi'}(\kp)^{k}$. We set $a_{\pi\times\pi'}(\kn)=0$ if $\kn$ is not a prime ideal power. It is known, by an argument due to Brumley (see \cite[Lemma 2.2 and Appendix]{ST-2019}), that for all $\kn\subseteq \mathcal{O}_F$,
	\begin{equation}\label{a-ineq}
		|a_{\pi\times\pi'}(\kn)|\leq \sqrt{a_{\pi\times\tilde{\pi}}(\kn)a_{\pi'\times{\tilde{\pi}'}}(\kn)}.
	\end{equation}
In particular, taking the trivial representation $\pi'\in\mathfrak{F}_{1}$ yields
\begin{equation}\label{coeffpair}
	|a_{\pi}(\kn)|^2\leq a_{\pi\times\tilde{\pi}}(\kn).
\end{equation}
An analogous inequality holds for the Dirichlet coefficients of the $L$-functions themselves. For all $\kn\subseteq \mathcal{O}_F$, 
\begin{equation}\label{ineq-22rs}
|\lambda_{\pi\times\pi'}(\kn)|\leq \sqrt{\lambda_{\pi\times\tilde{\pi}}(\kn)\lambda_{\pi'\times{\tilde{\pi}'}}(\kn)}.
\end{equation}
See the details in our joint work with L\"u and Wang \cite[Lemma 3.1]{JLW-2020}. In particular, when $\pi'\in\mathfrak{F}_{1}$ is the trivial representation, this implies
\begin{equation}\label{ineq-2rs}
    |\lambda_{\pi}(\kn)|^2\leq \lambda_{\pi\times\tilde{\pi}}(\kn).
\end{equation}

We now establish an inequality between the coefficients of the Rankin--Selberg $L$-function and those of its reciprocal. This result may be viewed as a complement to the inequalities \eqref{a-ineq} and \eqref{ineq-22rs}. For $\Re(s)>1$, define $\mu_{\pi}(\kn)$ and $\mu_{\pi\times\pi'}(\kn)$ by
\[
L(s,\pi)^{-1}=\sum_{\kn}\frac{\mu_{\pi}(\kn)}{\N\kn^s},
\qquad
L(s,\pi\times\pi')^{-1}=\sum_{\kn}\frac{\mu_{\pi\times\pi'}(\kn)}{\N\kn^s}.
\]

\begin{lemma}\label{lem-ineq}
Let $\pi\in\mathfrak{F}_n$ and $\pi'\in\mathfrak{F}_{n'}$. Then the inequality
\begin{equation}\label{f-sec-rsmm}
|\mu_{\pi\times\pi'}(\kn)|\leq \sqrt{\lambda_{\pi\times\tilde{\pi}}(\kn)\lambda_{\pi'\times{\tilde{\pi}'}}(\kn)}
\end{equation}
holds for any integral ideal $\kn$. In particular, for any $\pi\in\mathfrak{F}_n$, we have
\begin{equation}\label{f-sec-rs}
|\mu_{\pi}(\kn)|^2\leq \lambda_{\pi\times\tilde{\pi}}(\kn),\quad \text{and}\quad |\mu_{\pi\times\tilde{\pi}}(\kn)|\leq \lambda_{\pi\times\tilde{\pi}}(\kn)
\end{equation}
for any integral ideal $\kn$.
\end{lemma}
\begin{remark}
The inequality \eqref{f-sec-rsmm} was previously established by Humphries and Thorner \cite{HT-2024} under the restriction that $(\kn, \kq_{\pi} \kq_{\pi'})=\mathcal{O}_F$. Here, we remove this coprimality condition. This generalization is important for the proofs of some results without any constraint on the conductors.
\end{remark}

\begin{proof}
To prove this lemma, we begin by recalling a useful recursive identity, which goes as follows: Let $\{\beta_j\}_{j=1}^n$ be a finite sequence of complex numbers. The elementary symmetric polynomials $e_\nu(\beta_1,\ldots, \beta_n)$ in variables $\{\beta_j\}_{j=1}^n$ are defined for $\nu\geq 0$ by
\[
e_\nu(\beta_1,\cdots, \beta_n)=\sum_{1\leq j_1< j_2<\cdots< j_\nu\leq n} \beta_{j_1}\beta_{j_2}\cdots\beta_{j_\nu},
\]
with the convention that $e_\nu(\beta_1,\cdots, \beta_n)=0$ for any $\nu>n$. Moreover, for $\nu\geq 1$, define the $\nu$-th power sum $c_\nu(\beta_1,\cdots, \beta_n)$ by 
\[
c_\nu(\beta_1,\cdots, \beta_n)=\sum_{j=1}^n\beta_j^\nu.
\]
Newton's identity provides a recursion connecting these two types of polynomials, which states 
\begin{equation}\label{recur-relation}
\nu e_\nu(\beta_1,\cdots, \beta_n)=\sum_{j=1}^{\nu}(-1)^{j-1}e_{\nu-j}(\beta_1,\cdots, \beta_n)c_j(\beta_1,\cdots, \beta_n)
\end{equation}
valid for $1\leq \nu\leq n$.

We shall apply this recursion to the following specific case. Consider the local factor of the Rankin--Selberg $L$-function, which can be expressed as
\[
L_{\kp}(s,\pi\times\pi')^{-1}=\sum_{\nu=0}^{\infty}\frac{\mu_{\pi\times\pi'}(\kp^{\nu})}{\N\kp^{\nu s}}=\prod_{j=1}^{n}\prod_{j'=1}^{n'}\Big(1-\frac{\alpha_{j,j',\pi\times\pi'}(\kp) }{\N\kp^{s}}\Big).
\]
By comparing the Euler product with the Dirichlet series expansion and applying Newton's identity \eqref{recur-relation} to the local parameters, we obtain from \eqref{def-api} the recursive relation
\begin{equation}\label{re2}
\nu\mu_{\pi\times\pi'}(\kp^\nu)=-\sum_{l=1}^\nu a_{\pi\times\pi'}(\kp^l)\mu_{\pi\times\pi'}(\kp^{\nu-l})
\end{equation}
for $1\leq \nu\leq nn'$ and $\mu_{\pi\times\pi'}(\kp^\nu)=0$ for $\nu>n n'$. Moreover, we recall another recursion for the coefficients of the Rankin--Selberg $L$-function:
\begin{equation}\label{re-3-lambda}
\nu\lambda_{\pi\times\pi'}(\kp^\nu)=\sum_{l=1}^\nu a_{\pi\times\pi'}(\kp^l)\lambda_{\pi\times\pi'}(\kp^{\nu-l})
\end{equation}
for $\nu\geq 1$ (see \cite[equation (3.14)]{JLW-2020}).

Since each Rankin--Selberg $L$-function admits the Euler product, we find the arithmetic function $\mu_{\pi\times\pi'}(\kn)$ is multiplicative. Hence, it suffices to establish  the desired inequality \eqref{f-sec-rsmm} for the powers of prime ideals, that is
\begin{equation}\label{p-power-ineq}
|\mu_{\pi\times\pi'}(\kp^\nu)|\leq \sqrt{\lambda_{\pi\times\tilde{\pi}}(\kp^\nu)\lambda_{\pi'\times{\tilde{\pi}'}}(\kp^\nu)}
\end{equation}
holds for all prime ideals $\kp$ and integers $\nu\ge 1$. This inequality holds trivially when $\nu>n n'$, since the left-hand side vanishes. We shall prove \eqref{p-power-ineq} by induction on the exponent $\nu$. For $\nu=1$, it follows from \eqref{re2} and \eqref{re-3-lambda} that
$\mu_{\pi\times\pi'}(\kp) = -a_{\pi\times\pi'}(\kp)$ and $\lambda_{\pi\times\tilde{\pi}}(\kp) = a_{\pi\times\tilde{\pi}}(\kp)$.
We then apply the inequality \eqref{a-ineq} to obtain
\[
|\mu_{\pi\times\pi'}(\kp)|^2=|a_{\pi\times\pi'}(\kp)|^2\leq a_{\pi\times\tilde{\pi}}(\kp) a_{\pi'\times{\tilde{\pi}'}}(\kp)=\lambda_{\pi\times\tilde{\pi}}(\kp)\lambda_{\pi'\times{\tilde{\pi}'}}(\kp).
\]
For any $1\leq \nu\leq n n'$, assume that \eqref{p-power-ineq} holds for all exponents up to $\nu$. That is
\begin{equation}\label{ineq-hy}
|\mu_{\pi\times\pi'}(\kp^c)|^2\leq \lambda_{\pi\times\tilde{\pi}}(\kp^c)\lambda_{\pi'\times{\tilde{\pi}'}}(\kp^c)
\quad \text{for all} \;c\leq \nu.
\end{equation}
By applying \eqref{re2} and the triangle inequality, we deduce
\begin{equation}\label{hy-step}
(\nu+1)|\mu_{\pi\times\pi'}(\kp^{\nu+1})|\leq \sum_{l=1}^{\nu+1} |a_{\pi\times\pi'}(\kp^l)\mu_{\pi\times\pi'}(\kp^{\nu+1-l})|.
\end{equation}
Further, by using \eqref{a-ineq}, the induction hypothesis \eqref{ineq-hy} and the Cauchy--Schwarz inequality, the sum on the right-hand side of \eqref{hy-step} is bounded by
\begin{equation*}
\begin{aligned}
&\sum_{l=1}^{\nu+1}\sqrt{a_{\pi\times\tilde{\pi}}(\kp^{l})a_{\pi'\times{\tilde{\pi}'}}(\kp^{l})\lambda_{\pi\times\tilde{\pi}}(\kp^{\nu+1-l})\lambda_{\pi'\times{\tilde{\pi}'}}(\kp^{\nu+1-l})}\\
    \leq& \left(\sum_{l=1}^{\nu+1}a_{\pi\times\tilde{\pi}}(\kp^{l})\lambda_{\pi\times\tilde{\pi}}(\kp^{\nu+1-l})\right)^{\frac{1}{2}}\left(\sum_{l=1}^{\nu+1}a_{\pi'\times{\tilde{\pi}'}}(\kp^{l})\lambda_{\pi'\times{\tilde{\pi}'}}(\kp^{\nu+1-l})\right)^{\frac{1}{2}},
\end{aligned}
\end{equation*}
which is just equal to
$
(\nu+1)\sqrt{\lambda_{\pi\times\tilde{\pi}}(\kp^{\nu+1})\lambda_{\pi'\times{\tilde{\pi}'}}(\kp^{\nu+1})}
$
by \eqref{re-3-lambda}. Hence, we conclude from \eqref{hy-step} that
\[
|\mu_{\pi\times\pi'}(\kp^{\nu+1})|\leq \sqrt{\lambda_{\pi\times\tilde{\pi}}(\kp^{\nu+1})\lambda_{\pi'\times{\tilde{\pi}'}}(\kp^{\nu+1})}.
\]
This completes the induction argument and establishes the claimed inequality \eqref{f-sec-rsmm}.

The two inequalities in \eqref{f-sec-rs} are in fact special cases of \eqref{f-sec-rsmm}. Indeed, they follow by taking the trivial representation $\pi' \in \mathfrak{F}_1$ for the first, and $\pi' \cong \tilde{\pi}$ for the second in \eqref{f-sec-rsmm}.
\end{proof}

\vskip 5mm 

\section{Power sieve for number fields}

\subsection{Construction of characters}

The power sieve originated from Heath-Brown’s square sieve \cite{HB-1984}, which was designed to identify squares over the rational field $\Q$ via quadratic characters such as the Legendre symbol. Munshi \cite{Munshi-2009} later generalized this concept to detect $k$-th powers, and established the power sieve. This method has important applications in counting rational points on algebraic varieties (see \cite{Munshi-2009, HB-P-2012}). 

In this work, we extend the power sieve to arbitrary number fields $F$, which broadens its applicability in analytic number theory. In particular, we apply the power sieve to problems involving automorphic forms. This generalization requires replacing detection of rational integers with an ideal-theoretic framework. While the analytic structure of the sieve remains essentially unchanged, the main challenge lies in detecting $k$-th power ideals. To address this, we construct primitive ray class characters whose restrictions to the local cokernels have exact order $k$. 

Let $F$ be a number field and let $k\geq 2$ be fixed. We denote by $\Cl(F)$ the ideal class group of $F$, and by $h_F=|\Cl(F)|$ its class number. Put $L = F\bigl(\zeta_k,\sqrt[k]{\mathcal O_F^\times}\bigr)$.
By Dirichlet's unit theorem, the unit group $\mathcal{O}_F^\times$ is finitely generated. Thus, $L$ is a finite Galois extension of $F$. We shall use the following consequence of the complete splitting condition.

\begin{lemma}\label{lem-local-cokernel-character}
Let $\kp$ be a nonzero prime ideal of $\mathcal O_F$ such that $\kp\nmid k\mathcal{O}_F$ and $\kp$ splits completely in $L$. Let $\phi_\kp: \mathcal O_F^\times \longrightarrow (\mathcal O_F/\kp)^\times$    
be the natural reduction map on units. Then the cokernel
\[
\operatorname{coker}(\phi_\kp) = (\mathcal O_F/\kp)^\times/ \phi_\kp(\mathcal O_F^\times)
\]
admits a character of exact order $k$.
\end{lemma}

\begin{proof}
Since $F(\zeta_k)\subseteq L$ and $\kp$ splits completely in $L$, it also splits completely in $F(\zeta_k)$. Let $\mathfrak P$ be a prime ideal of $\mathcal O_{F(\zeta_k)}$ lying above $\kp$. Then the residue class
degree of $\mathfrak P/\kp$ is $1$, and hence the natural map $\mathcal O_F/\kp \longrightarrow
   \mathcal O_{F(\zeta_k)}/\mathfrak P$
is an isomorphism.

Reducing the factorization
\[
   X^k-1=\prod_{j=0}^{k-1}(X-\zeta_k^j)
\]
modulo $\mathfrak P$, and using the above identification of residue fields, we see that $X^k-1$ splits completely over $\mathcal O_F/\kp$. Since $\kp\nmid k\mathcal O_F$, the polynomial $X^k-1$ is separable over $\mathcal O_F/\kp$. Thus, $(\mathcal O_F/\kp)^\times$ contains all distinct $k$-th roots of unity. It gives $k\mid (N\kp-1)$. Since $(\mathcal O_F/\kp)^\times$ is cyclic, it follows that
\begin{equation}\label{eq:OFp/kpower}
   \left|
   (\mathcal O_F/\kp)^\times/
   \bigl((\mathcal O_F/\kp)^\times\bigr)^k
   \right|
   =(k, N\kp-1)=k.
\end{equation}
The quotient group in \eqref{eq:OFp/kpower} is also cyclic.

Next, let $u\in\mathcal O_F^\times$. By definition, we may choose $\alpha_u\in \mathcal O_L$ such that
$\alpha_u^k=u$. Let $\mathfrak Q$ be a prime ideal of $\mathcal O_L$ lying above $\kp$. Since $\kp$ splits completely in $L$, the residue class degree $f(\mathfrak Q/\kp)$ is $1$. Thus, the map 
$\mathcal O_F/\kp \longrightarrow \mathcal O_L/\mathfrak Q$
is an isomorphism. Reducing $\alpha_u^k=u$ modulo $\mathfrak Q$, and identifying the two residue fields through the above isomorphism, we obtain an element $a_u\in\mathcal O_F/\kp$ such that $a_u^k=\phi_\kp(u)$. This holds for any $u\in\mathcal O_F^\times$. Thus, we conclude that
\begin{equation}\label{eq:unit-image-tok}
  \phi_\kp(\mathcal O_F^\times)
   \subseteq
   \bigl((\mathcal O_F/\kp)^\times\bigr)^k.  
\end{equation}

It follows from \eqref{eq:unit-image-tok} that there is a surjective homomorphism
\[
   \coker(\phi_\kp)
   \longrightarrow
   (\mathcal O_F/\kp)^\times/
   \bigl((\mathcal O_F/\kp)^\times\bigr)^k.
\]
By \eqref{eq:OFp/kpower}, the latter group is cyclic of order $k$, and then admits a complex character of exact order $k$. Since the above homomorphism is surjective, composing this character with it yields
a complex character of $\coker(\phi_\kp)$ of exact order $k$.
\end{proof}

\vskip 2mm

Now let $P$ be a large real parameter. Define the set of prime ideals
\[
\mathcal{P} := \left\{ \kp \subset \mathcal{O}_F : \N\kp \leq P,\; \kp\nmid k\mathcal{O}_F,\; \kp \text{ splits completely in } L\right\}.
\]
For each $\kp\in {\mathcal P}$, let $\Cl_\kp(F)$ denote the ray class group modulo $\kp$, and let $\rho_\kp:\Cl_\kp(F)\to \Cl(F)$ be the natural map. By weak approximation (see \cite[Corollary~1.2.11]{Cohen-2000}), every ideal class in $\Cl(F)$ has an integral ideal representative coprime to $\kp$. Hence, $\rho_\kp$ is surjective. Its kernel is identified with $\coker(\phi_\kp) = (\mathcal O_F/\kp)^\times/ \phi_\kp(\mathcal O_F^\times)$. Thus, we have the short exact sequence
\[
1 \longrightarrow \operatorname{coker}(\phi_\kp) \longrightarrow \Cl_\kp(F) \xrightarrow{\rho_\kp} \Cl(F) \longrightarrow 1.
\]
See also the proof of \cite[Proposition 3.2.3]{Cohen-2000}.

By Lemma \ref{lem-local-cokernel-character}, we choose a character of exact order $k$
\[
\theta_\kp: \coker(\phi_\kp) \longrightarrow \mathbb C^\times.
\]
Since all groups involved are finite abelian, it follows from Pontryagin duality that $\theta_\kp$ extends to a character
\[
\chi_\kp: \Cl_\kp(F) \longrightarrow \mathbb C^\times.
\]
Note that $\chi_\kp$ is primitive, because $\chi_\kp$ is non-trivial on $\coker(\phi_\kp)$ and the modulus $\kp$ is prime.

The $k$-th power $\chi_\kp^k$ is trivial on $\coker(\phi_\kp)$. Hence it factors through $\rho_\kp:\Cl_\kp(F)\to \Cl(F)$. Thus, there exists $\eta_\kp\in\widehat{\Cl(F)}$ such that
\[
\chi_\kp^k = \eta_\kp \circ \rho_\kp,
\]
where $\widehat{\Cl(F)}$ denotes the Pontryagin dual of $\Cl(F)$.

We next choose a subfamily of ${\mathcal{P}}$ for which the character $\eta_\kp$ is independent of $\kp$. For each $\eta\in\widehat{\Cl(F)}$, define ${\mathcal P}_\eta = \{\kp\in\mathcal P:\eta_\kp=\eta\}$.
Then one has 
\[
\sum_{\eta\in\widehat{\Cl(F)}} \sum_{\kp\in\mathcal P_\eta}\log N\kp = \sum_{\kp\in\mathcal P}\log N\kp.
\]
Since $|\widehat{\Cl(F)}|=h_F$, there exists a character $\eta^\ast=\eta^\ast(P)\in\widehat{\Cl(F)}$ such that
\begin{equation}\label{eq:Wsum-starlower}
 \sum_{\kp \in {\mathcal P}_{\eta^\ast}}\log \N\kp \geq \frac{1}{h_F} \sum_{\kp\in\mathcal P}\log N\kp.   
\end{equation}
Note that for all $\kp\in {\mathcal P}_{\eta^\ast}$, we have $\eta_\kp=\eta^\ast$, which further gives
$\chi_\kp^k = \eta^\ast \circ\rho_\kp$.

We extend $\chi_\kp$ to all nonzero integral ideals by setting
\[
\chi_\kp(\mathfrak a)=0 \qquad \text{if } \kp\mid\mathfrak a.
\]
If $(\mathfrak a,\kp)=1$, then $\chi_\kp(\mathfrak a)$ denotes the value of $\chi_\kp$ on the ray class of $\mathfrak a$. We now record the key detection property. Let $\mathfrak b$ be an integral ideal, and its class in $\Cl(F)$ denote by $[\mathfrak b]$. If $\kp\in {\mathcal P}_{\eta^\ast}$ and $\kp\nmid\mathfrak b$, then
\begin{equation}\label{eq:detect-prop}
    \chi_\kp(\mathfrak b^k) = \chi_\kp^k(\mathfrak b) = \eta^\ast([\mathfrak b]).
\end{equation}
Since $\eta^\ast$ is a character of $\Cl(F)$, we have $|\eta^\ast([\mathfrak b])|=1$. This detection property for $k$-th powers is the analytic core of the sieve.

\subsection{Power sieve inequality}
In our context, the key ingredient we extract from the power sieve is the following lemma, which provides an upper bound for the contribution of $k$-th powers weighted by a general arithmetic sequence.

\begin{lemma}\label{lem-powersieve}
Let $\{f(\kn)\}_{\kn}$ be a finite sequence of non-negative numbers indexed by nonzero ideals $\kn \subseteq \mathcal{O}_F$. Let $k \geq 2$ be a fixed integer. Then for any nonzero ideal $\mathfrak{a}\subseteq \mathcal{O}_F$, we have
\[
  \sum_{\kn} f(\kn^k) \ll P^{-2} D_{F}^{70k^{[F:\Q]+1}}\sum_{\kn} f(\kn) \Bigg(\Big| \sum_{\substack{\kp \in {\mathcal P}_{\eta^\ast} \\ (\kp, \mathfrak{a})=\mathcal{O}_F}} \chi_{\mathfrak{p}}(\kn) \log \N\kp \Big|^2 + \log^2(2+\N\mathfrak{a}\kn)\Bigg),
\]
where ${\mathcal P}_{\eta^\ast}$ is the set of prime ideals defined above with norm bound $P$, and for each $\kp \in {\mathcal P}_{\eta^\ast}$, $\chi_\kp$ is the corresponding primitive ray class character constructed above. The implied constant depends only on $k$ and $[F:\Q]$.
\end{lemma}

\begin{remark}
This estimate provides an upper bound for a sum weighted by $f(\kn)$ over $k$-th powers. In applications, $f(\kn)$ is typically a structured arithmetic function of interest (for example, the coefficients of an $L$-function). The power of this method lies in transforming the original problem into one of bounding character-twisted sums over $\kn$, which often exhibit crucial cancellation and are therefore smaller than the trivial estimate $\sum_{\kn} f(\kn)$.
\end{remark}

\begin{proof}
Since the weight terms in the parentheses on the right-hand side of the inequality are non-negative, it suffices to show that they are $\gg P^2D_{F}^{-70k^{[F:\Q]+1}}$ when $\kn$ is the $k$-th power of some ideal $\mathfrak{b}$. Specifically, our goal is to establish that for any ideal $\mathfrak{b}$, the inequality
\[
\Big| \sum_{\substack{\kp \in {\mathcal P}_{\eta^\ast}\\ (\kp, \mathfrak{a})=\mathcal{O}_F}} \chi_{\mathfrak{p}}(\kn) \log \N\kp \Big|^2 + \log^2(\N\mathfrak{a}\kn)\gg_{k,[F:\Q]} \frac{P^2}{D_{F}^{70k^{[F:\Q]+1}}}
\]
holds for $\kn= \mathfrak{b}^k$, provided that $P\gg D_{F}^{35k^{[F:\Q]+1}}$. 

If $\kn= \mathfrak{b}^k$, then $\chi_{\kp}(\kn) = \eta^\ast([\mathfrak b])$ for all $\kp \in {\mathcal P}_{\eta^\ast}$ not dividing $\mathfrak{b}$ by \eqref{eq:detect-prop}. Hence, we find
\[
\begin{aligned}
  \Big| \sum_{\substack{\kp \in {\mathcal P}_{\eta^\ast} \\ (\kp, \mathfrak{a})=\mathcal{O}_F}} \chi_{\mathfrak{p}}(\kn) \log \N\kp \Big|^2 + \log^2(\N\mathfrak{a}\kn)=&\Big(\sum_{\substack{\kp \in {\mathcal P}_{\eta^\ast} \\ (\kp, \mathfrak{a}\kn)=\mathcal{O}_F}}\log \N\kp \Big)^2 + \log^2(\N\mathfrak{a}\kn)\\
\gg&\Big(\sum_{\substack{\kp \in {\mathcal P}_{\eta^\ast} \\ (\kp, \mathfrak{a}\kn)=\mathcal{O}_F}}\log \N\kp+ \log(\N\mathfrak{a}\kn)\Big)^2\\
    \gg &\Big(\sum_{\kp \in {\mathcal P}_{\eta^\ast}}\log \N\kp
    \Big)^2,  
\end{aligned}
\]
where the final bound uses the inequality $\sum_{\kp|\mathfrak{a}\kn} \log \N\kp\leq \log(\N\mathfrak{a}\kn)$ (with equality if and only if $\mathfrak{a}\kn$ is square-free). We know from \cite[Chapter V, Theorem 4]{Lang-1994} that 
\[
h_F\leq \sum_{\N \kn\leq M_F} 1 \ll_{[F:\Q],\varepsilon} D_{F}^{\frac{1}{2}+\varepsilon}
\]
for any $\varepsilon>0$, where $M_F$ is Minkowski's bound satisfying the crude estimate $M_F\leq \sqrt{D_{F}}$. Thus, combining \eqref{eq:Wsum-starlower} with the estimates above, we obtain
\begin{equation*}
  \Big| \sum_{\substack{\kp \in {\mathcal P}_{\eta^\ast} \\ (\kp, \mathfrak{a})=\mathcal{O}_F}} \chi_{\mathfrak{p}}(\kn) \log \N\kp \Big|^2 + \log^2(\N\mathfrak{a}\kn)\gg  D_{F}^{-2}\Big(\sum_{\kp \in {\mathcal P}}\log \N\kp
        \Big)^2,
\end{equation*}
where $\mathcal{P}$ consists of prime ideals $\kp$ splitting completely in $L$ with $\N\kp \leq P$ and $\kp\nmid k\mathcal{O}_F$. The sum can thus be estimated using an effective version of the Chebotarev density theorem. Our final goal is to prove the key lower bound
\begin{equation} \label{eq-lower-bound}
\sum_{\substack{\mathrm{N}\mathfrak{p} \leq P \\ \mathfrak{p} \text{ splits completely in } L}} \log \mathrm{N}\mathfrak{p}  \gg \frac{P}{D_{F}^{35k^{[F:\Q]+1}-2}} \quad \text{for } P\gg D_{F}^{35k^{[F:\Q]+1}}.
\end{equation}
Note that excluding the prime ideals dividing \(k\mathcal O_F\) does not affect this lower bound due to $\sum_{\kp\mid k\mathcal O_F}\log \N\kp \ll_{k,[F:\mathbb Q]}1$.

To prove this, we apply the effective Chebotarev density theorem of Zaman \cite[Theorem 1.3.1]{Zaman-2017}, which yields that for $P \gg D_L^{35}$, 
\begin{equation}\label{eq-split-lower}
 \sum_{\substack{\mathrm{N}\mathfrak{p} \leq P \\ \mathfrak{p} \text{ splits completely in } L}} \log \mathrm{N}\mathfrak{p} \gg \frac{P}{D_L^{19}[L:F]}.
\end{equation}

To further simplify the bound, we relate the discriminant $D_L$ to the base field discriminant $D_F$ via the standard identity \cite[Chapter III, Corollary 2.10]{Neukirch-1999}
\[
D_L = D_F^{[L:F]} \cdot \N\mathcal{D}_{L/F},
\]
where $\mathcal{D}_{L/F}$ is the relative discriminant of the extension $L/F$. In the case $L = F(\zeta_k, \sqrt[k]{\mathcal{O}_F^\times})$, the norm $\N\mathcal{D}_{L/F}\ll_{k,[F:\mathbb{Q}]}1$. This yields the inequality
\[
D_L \ll_{k,[F:\mathbb{Q}]} D_F^{[L:F]}.
\]
For the degree of the extension, we have
\[
[L:F] \le [F(\zeta_k):F] \cdot [F(\sqrt[k]{\mathcal{O}_F^{\smash{\times}}}):F].
\]
We next estimate each term separately. The degree of the cyclotomic extension satisfies $[F(\zeta_k):F]\leq \varphi(k)\leq k$. By Dirichlet's Unit Theorem, the rank of the unit group $\mathcal{O}_F^\times$ is $r = r_1 + r_2 - 1$, where $r_1$ and $2r_2$ are the number of real and complex embeddings of $F$, respectively. So the group $\mathcal{O}_F^\times$ is generated by $r+1$ elements (one generator for the torsion part and $r$ fundamental units). Adjoining the $k$-th root of each of these generators contributes a factor of at most $k$ to the total extension degree. This gives the bound
\[
[F(\sqrt[k]{\mathcal{O}_F^{\smash{\times}}}):F] \leq k^{r+1}.
\]
Combining these gives an upper bound for the degree of the extension
\[
[L:F] \leq k \cdot k^{r+1} = k^{r+2}.
\]
Using the fact that $r+1 = r_1+r_2 \leq r_1+2r_2 = [F:\mathbb{Q}]$, we have $r+2\leq [F:\mathbb{Q}]+1$. Therefore, 
\[
[L:F] \leq k^{r+2} \leq k^{[F:\mathbb{Q}]+1}.
\]
Substituting the estimates for $[L:F]$ and $D_L$ into \eqref{eq-split-lower} yields a lower bound of the form $\gg_{k,[F:\mathbb{Q}]} P / D_F^{19k^{[F:\Q]+1}}$. Since a weaker bound suffices for our application, we may relax the exponent, and obtain
\[
\sum_{\substack{\mathrm{N}\mathfrak{p} \leq P \\ \mathfrak{p} \text{ splits completely in } L}} \log \mathrm{N}\mathfrak{p}\gg_{k,[F:\mathbb{Q}]} \frac{P}{D_F^{35k^{[F:\Q]+1}-2}},
\]
which establishes the desired bound in \eqref{eq-lower-bound} and completes the proof.
\end{proof}

\section{Proof of Theorem \ref{thm-main}}
In this section, we apply the power sieve over number fields to study the distribution of Rankin--Selberg coefficients. To the best of our knowledge, our work presents the first application of the power sieve in the context of automorphic forms.

\subsection{Initial estimate from the power sieve}
Our goal is to estimate the sum over the $k$-th powers
\[
S_{\pi}(X) = \sum_{\kn} \lambda_{\pi\times\tilde{\pi}}(\kn^k)V\Big(\frac{\N\kn^k}{X}\Big)
\]
for any fixed $k\geq 2$, where $V:\R \rightarrow [0,\infty)$ is a smooth function compactly supported on $[1/2,5/2]$ which is equal to $1$ on the interval $[1,2]$, and satisfies $V^{(j)}\ll_j 1$ for all $j\geq 1$. 
By partial integration, the Mellin transform of $V$ satisfies
\begin{equation}\label{mellin}
\hat{V}(s)=\int_{0}^{\infty} V(x)x^{s-1}\dd x=\frac{(-1)^j}{s(s+1)\cdots(s+j-1)}\int_{0}^{\infty} V^{(j)}(x)x^{s+j-1} \dd x  \ll |s|^{-j}
\end{equation}
for any $j\geq 1$ and $\varepsilon \leq \sigma=\Re s \leq 2$.

We apply Lemma \ref{lem-powersieve} with $f(\kn)$ to be $\lambda_{\pi\times\tilde{\pi}}(\kn)V(\N\kn/X)$, and the parameter $\mathfrak{a}$ to be the arithmetic conductor $\kq_{\pi}$ of $\pi$, and then get 
\begin{equation}\label{eq-Spi-tochar}
S_{\pi}(X) \ll  \frac{D_{F}^{70k^{[F:\Q]+1}}}{P^{2}}\sum_{\kn}\lambda_{\pi\times\tilde{\pi}}(\kn)V\Big(\frac{\N\kn}{X}\Big) \Bigg( \Big|\sum_{\substack{\kp \in {\mathcal P}_{\eta^\ast} \\ (\kp, \kq_{\pi})=\mathcal{O}_F}} \chi_{\kp}(\kn) \log \N\kp
\Big|^2+\log^2(2+\N\kq_{\pi}\kn)\Bigg).
\end{equation}
Due to the estimate \eqref{eq-convexity-2} in Lemma \ref{lem-convexity}, Rankin's trick yields
\begin{equation}\label{eq-Libound}
\sum_{\N\kn\leq X}\lambda_{\pi\times\tilde{\pi}}(\kn)\ll X\sum_{\N\kn\leq X}\frac{\lambda_{\pi\times\tilde{\pi}}(\kn)}{\N\kn^{1+\frac{1}{\log X}}}\ll X L\Big(1+\frac{1}{\log X},\pi\times\tilde{\pi}\Big) \ll C(\pi)^\varepsilon X\log X.
\end{equation}
Thus, the contributions of the term $\log^2(2+\N\kq_{\pi}\kn)$ is 
\begin{equation}\label{eq-contr-sec}
\ll \Big(\log^2(X\N\kq_{\pi})\Big)\sum_{\kn}\lambda_{\pi\times\tilde{\pi}}(\kn)V\Big(\frac{\N\kn}{X}\Big)\ll C(\pi)^\varepsilon X\log^3 X.
\end{equation}
We next express 
\[
S = \sum_{\kn} \lambda_{\pi\times\tilde{\pi}}(\kn)V\Big(\frac{\N\kn}{X}\Big)\Big|\sum_{\substack{\kp \in {\mathcal P}_{\eta^\ast} \\ (\kp, \kq_{\pi})=\mathcal{O}_F}} \chi_{\mathfrak{p}}(\kn) \log \N\kp
\Big|^2.
\]
Expanding $S$, we get
\[
\begin{aligned}
S &= \sum_{\substack{\kp,\kq \in {\mathcal P}_{\eta^\ast} \\ (\kp\kq, \kq_{\pi})=\mathcal{O}_F}}\log \N\kp \log \N\kq \sum_{\kn} \lambda_{\pi\times\tilde{\pi}}(\kn) \chi_{\kp}(\kn) \overline{\chi_{\kq}(\kn)}V\Big(\frac{\N\kn}{X}\Big) \\
&\ll P (\log P) \sum_{\kn} \lambda_{\pi\times\tilde{\pi}}(\kn) V\Big(\frac{\N\kn}{X}\Big)+ P^2 \max_{\substack{\kp,\kq\in {\mathcal P}_{\eta^\ast}\\\kp\neq\kq\\(\kp\kq,\kq_{\pi})=1}} \bigg| \sum_{\kn} \lambda_{\pi\times\tilde{\pi}}(\kn) \psi_{\kp\kq}(\kn)  V\Big(\frac{\N\kn}{X}\Big)\bigg|,
\end{aligned}
\]
where $\psi_{\kp\kq}=\chi_{\kp} \overline{\chi_{\kq}}$ is a primitive ray class character modulo $\kp\kq$. Next, it suffices to estimate the sum in absolute value on the right-hand side. 

By Mellin's inverse transform, we can write
\begin{equation}\label{eq-mellin}
 \sum_{\kn} \lambda_{\pi\times\tilde{\pi}}(\kn) \psi_{\kp\kq}(\kn)  V\Big(\frac{\N\kn}{X}\Big)=\frac{1}{2\pi i}\int_{(2)}L(s,\pi\times\tilde{\pi}\otimes\psi_{\kp\kq})X^{s}\hat{V}(s)\dd s.
\end{equation}
Note that $L(s,\pi\times\tilde{\pi}\otimes\psi_{\kp\kq})$ is entire. Moving the vertical line of integration in \eqref{eq-mellin} to $\Re s=\varepsilon$, we obtain from Cauchy's theorem, \eqref{mellin} and the convexity bound \eqref{eq-twist-convexity} that
\[
\sum_{\kn} \lambda_{\pi\times\tilde{\pi}}(\kn) \psi_{\kp\kq}(\kn)  V\Big(\frac{\N\kn}{X}\Big)\ll P^{n^2}C(\pi)^{n}(P C(\pi) X)^{\varepsilon}.
\]                                                              
Together with the estimate \eqref{eq-Libound}, we get
\[
S\ll (PX+ P^{n^2+2}C(\pi)^{n})(P C(\pi) X)^{\varepsilon}.
\]
Inserting this and \eqref{eq-contr-sec} into \eqref{eq-Spi-tochar}, we derive the crude estimate
\[
S_{\pi}(X) \ll D_{F}^{70k^{[F:\Q]+1}}(P^{-1}X+ P^{n^2}C(\pi)^{n})(PC(\pi)X)^{\varepsilon} \ll C(\pi)^{n+70k^{[F:\Q]+1}+\varepsilon} X^{1-\frac{1}{n^2+1}+2\varepsilon}
\]
on taking $P=X^{1/(n^2+1)}$ and using the fact that $D_F\leq C(\pi)$ by definition. By dyadic decomposition and the definition of $V$, we then get
\[
\sum_{\N\kn\leq X} \lambda_{\pi\times\tilde{\pi}}(\kn^k)\ll C(\pi)^{n+70k^{[F:\Q]+1}+\varepsilon} X^{k(1-\frac{1}{n^2+1}+3\varepsilon)}. 
\]
Therefore, this yields from partial summation that for any fixed $k\geq 2$,
\begin{equation}\label{eq-allupper}
\sum_{\kn}\frac{\lambda_{\pi\times\tilde{\pi}}(\kn^k)}{\N\kn^{k\sigma_0}}\ll C(\pi)^{n+70k^{[F:\Q]+1}+\varepsilon}
\end{equation}
with $\sigma_0=1-\frac{1}{n^2+1}+4\varepsilon$. We note that the estimate \eqref{eq-allupper}, together with \eqref{coeffpair} and \eqref{re-3-lambda}, already implies Hypothesis~H.

\subsection{Iterative process}
We follow an idea of Iwaniec \cite{Iwaniec-1990} to iteratively reduce the exponent in the $C(\pi)$-aspect. For convenience, we restrict our attention to sums over square-free integral ideals, which simplifies the argument and facilitates the iteration process. We define the $F$-M\"obius function $\mu_{F}$ via 
\begin{equation*}
\mu_{F}(\kn)=\left\{\begin{aligned}
&(-1)^{r}  \;\;\;\,\text{ if }\kn=\kp_{1}\kp_{2}\cdots \kp_{r} \text{ for some }r\in\mathbb{N}, \\
&1 \,\;\;\;\quad\quad  \text{ if }\kn=\mathcal{O}_F,\\
&0 \,\;\;\;\quad\quad  \text{ otherwise.}
\end{aligned}
\right.
\end{equation*}
Thus $\mu_{F}$ is multiplicative, supported on nonzero, square-free integral ideals. We get from \eqref{eq-allupper} the initial bound 
\begin{equation}\label{eq-muupper}
\sum_{\kn}\frac{\mu_{F}^2(\kn)\lambda_{\pi\times\tilde{\pi}}(\kn^k)}{\N\kn^{k\sigma_0}}\ll C(\pi)^{n+70k^{[F:\Q]+1}+\varepsilon}.
\end{equation}
By multiplicativity of the function $\lambda_{\pi\times\tilde{\pi}}(\kn)$, one has, for $\kn_1$ and $\kn_2$ two square-free integral ideals, 
\[
\lambda_{\pi\times\tilde{\pi}}(\kn_{1}^k)\lambda_{\pi\times\tilde{\pi}}(\kn_{2}^k)=\lambda_{\pi\times\tilde{\pi}}(\kd^{k})^2 \lambda_{\pi\times\tilde{\pi}}\left(\frac{\kn_{1}^k \kn_{2}^k}{\kd^{2k}} \right),
\]
where $\kd=\left(\kn_1, \kn_2\right)$. Consider the Dirichlet series 
\[
D(k\sigma)=\sum_{\kn}\frac{\mu_{F}^2(\kn)\lambda_{\pi\times\tilde{\pi}}(\kn^k)}{\N\kn^{k\sigma}}
\]
for $\sigma\geq \sigma_0$. Hence, we deduce from \eqref{eq-convexity-2} and \eqref{eq-LRS-2} that
\[
\begin{aligned}
D(k\sigma)^2\leq& \left(\sum_{\kd} \frac{\mu_{F}^2(\kd)\lambda_{\pi\times\tilde{\pi}}(\kd^{k})^2}{\N\kd^{2k\sigma}}\right)\sum_{\kn}\frac{\mu_{F}^2(\kn) \,d(\kn)\lambda_{\pi\times\tilde{\pi}}(\kn^k)}{\N\kn^{k\sigma}} \\
\leq& D(k\sigma-\varepsilon)\sum_{\kd} \frac{\mu_{F}^2(\kd)\lambda_{\pi\times\tilde{\pi}}(\kd^{k})^2}{\N\kd^{2k\sigma_0}}\\
	\ll& D(k\sigma-\varepsilon)\sum_{\kd} \frac{\mu_{F}^2(\kd)\lambda_{\pi\times\tilde{\pi}}(\kd^{k})}{\N\kd^{k+\varepsilon/2}}\cdot \frac{\lambda_{\pi\times\tilde{\pi}}(\kd^{k})}{\N\kd^{k(1-\frac{2}{n^2+1})+\varepsilon/2}}  \\
	\ll& D(k\sigma-\varepsilon)\sum_{\kd} \frac{\lambda_{\pi\times\tilde{\pi}}(\kd)}{\N\kd^{1+\varepsilon/2k}}\ll C(\pi)^{\varepsilon}D(k\sigma-\varepsilon),
\end{aligned}
\]
where $d(\kn)$ denotes the number of integral ideal divisors of $\kn$, and satisfies $d(\kn)\ll \N\kn^{\varepsilon}$ for any $\varepsilon>0$.
By repeating this process, we arrive at the general form of the iterative inequality
\[
D(k\sigma)^{2^l}\ll C(\pi)^{(2^l-1)\varepsilon}D(k\sigma-l\varepsilon)\ll C(\pi)^{n+70k^{[F:\Q]+1}+2^l\varepsilon}
\]
for any positive integer $l$ and $\sigma-l\varepsilon/k\geq \sigma_0$, where the latter inequality is deduced from \eqref{eq-muupper}. Taking the root of degree $2^l$ with $l=\lceil \log_2((n+70k^{[F:\Q]+1})/\varepsilon) \rceil=O(\varepsilon^{-1/2})$, we obtain the bound 
\[
\sum_{\kn}\frac{\mu_{F}^2(\kn)\lambda_{\pi\times\tilde{\pi}}(\kn^k)}{\N\kn^{k\sigma}}\ll C(\pi)^{2\varepsilon}
\]
for any $\sigma\geq \sigma_0+\sqrt{\varepsilon}$. By the multiplicativity of $\lambda_{\pi\times\tilde{\pi}}(\kn)$, we further deduce  
    \begin{equation}\label{eq-prod-fixk}
        \prod_{\kp}\left(1+ \frac{\lambda_{\pi\times\tilde{\pi}}(\kp^k)}{\N\kp^{k\sigma}}\right)\ll C(\pi)^{2\varepsilon}
    \end{equation}
for any fixed $k\geq 2$ and any $\sigma\geq \sigma_0+\sqrt{\varepsilon}$. 

\subsection{Final conclusions}
Applying the bound \eqref{eq-LRS-2}, we derive
\[
\lambda_{\pi\times\tilde{\pi}}(\kp^k)\N\kp^{-k\sigma_0}\ll \N\kp^{k\big(1-\frac{2}{n^2+1}+\varepsilon-\sigma_0\big)}\ll \N\kp^{-\frac{k}{n^2+1}}
\]
for any $k\geq n^2+2$. Summing over such $k$, it follows that
\[
\sum_{k\geq n^2+2}\frac{\lambda_{\pi\times\tilde{\pi}}(\kp^k)}{\N\kp^{k\sigma_0}}\ll \sum_{k\geq n^2+2} \N\kp^{-\frac{k}{n^2+1}}\ll \N\kp^{-1-\frac{1}{n^2+1}}.
\]
Combining this with the estimate \eqref{eq-prod-fixk}, we deduce
\[
\begin{aligned}
	\prod_{\kp}\left(1+\sum_{k=2}^{\infty} \frac{\lambda_{\pi\times\tilde{\pi}}(\kp^k)}{\N\kp^{k\sigma}}\right)\ll& \prod_{\kp}\left(1+\sum_{2\leq k\leq n^2+1} \frac{\lambda_{\pi\times\tilde{\pi}}(\kp^k)}{\N\kp^{k\sigma}}+O\Big(\N\kp^{-1-\frac{1}{n^2+1}}\Big)\right)\\
	\ll& \prod_{2\leq k\leq n^2+1}\prod_{\kp}\left(1+ \frac{\lambda_{\pi\times\tilde{\pi}}(\kp^k)}{\N\kp^{k\sigma}}\right)\\
	\ll& C(\pi)^{2n^2\varepsilon},
\end{aligned}
\]
for any $\sigma\geq \sigma_0+\sqrt{\varepsilon}$. After noticing $2n^2\varepsilon\leq \sqrt{\varepsilon}$ for sufficiently small $\varepsilon$, and then replacing $4\varepsilon+\sqrt{\varepsilon}$ by $\varepsilon$, the proof of the first assertion in Theorem \ref{thm-main} is complete.

We now turn to the second assertion, concerning the coefficients $a_{\pi\times\tilde{\pi}}(\kn)$. By a similar argument and again using \eqref{eq-LRS-2}, we find that
\[
\sum_{k\geq n^2+2}\frac{a_{\pi\times\tilde{\pi}}(\kp^k)}{\N\kp^{k\sigma}}\ll \sum_{k\geq n^2+2} \N\kp^{-\frac{k}{n^2+1}}\ll \N\kp^{-1-\frac{1}{n^2+1}}
\]
for any $\sigma \geq 1-\frac{1}{n^2+1}+\varepsilon$. Moreover, by the recurrence relation \eqref{re-3-lambda}, we have
\[
a_{\pi\times\tilde{\pi}}(\kp^k)\leq (n^2+1)\lambda_{\pi\times\tilde{\pi}}(\kp^k)
\]
for any $k\leq n^2+1$. Therefore, we get from the first assertion that 
\[
\begin{aligned}
	\prod_{\kp}\left(1+\sum_{k=2}^{\infty} \frac{a_{\pi\times\tilde{\pi}}(\kp^k)}{\N\kp^{k\sigma}}\right)\ll& \prod_{\kp}\left(1+(n^2+1)\sum_{2\leq k\leq n^2+1} \frac{\lambda_{\pi\times\tilde{\pi}}(\kp^k)}{\N\kp^{k\sigma}}+O\Big(\N\kp^{-1-\frac{1}{n^2+1}}\Big)\right)\\
	\ll& \prod_{\kp}\left(1+ \sum_{2\leq k\leq n^2+1}\frac{\lambda_{\pi\times\tilde{\pi}}(\kp^k)}{\N\kp^{k\sigma}}\right)^{n^2+1}\\
	\ll& C(\pi)^{(n^2+1)\varepsilon}.
\end{aligned}
\]
for any $\sigma \geq 1-\frac{1}{n^2+1}+\varepsilon$. The second assertion is derived from relabeling the constant $(n^2+1)\varepsilon$ as a new, arbitrarily small $\varepsilon$.
With both assertions established, the proof of Theorem \ref{thm-main} is now complete.
\vskip 5mm 

\section{Unconditional Correlation of Zeros of Automorphic $L$-functions}

Our Theorem \ref{thm-HH}, which establishes Hypothesis H for automorphic representations of arbitrary degree over any number field $F$, has numerous important applications in analytic number theory. As a primary and foundational application, we demonstrate how our result impacts the seminal work of Rudnick and Sarnak on the distribution of zeros of automorphic $L$-functions. While our proof of Hypothesis H is general, we will restrict the discussion to the setting of the rational field $\Q$ to adhere to the original framework of Rudnick and Sarnak.

In their 1996 paper \cite{RS-1996}, extending the results of Montgomery, Rudnick and Sarnak investigated the fine structure of the distribution of non-trivial zeros of primitive principal $L$-functions. These $L$-functions, denoted $L(s, \pi)$, are associated with cuspidal automorphic representations of $\mathrm{GL}_n(\mathbb{A}_{\Q})$. The non-trivial zeros of $L(s, \pi)$, denoted $\rho_{\pi}$, are located in the critical strip. Following Rudnick and Sarnak, we write them as $\rho_{\pi} = 1/2 + i\gamma_{\pi}$, which defines $\gamma_{\pi} := -i(\rho_{\pi} - 1/2)$. If the Generalized Riemann Hypothesis (GRH) for $L(s, \pi)$ holds, all $\gamma_{\pi}$ are real numbers. Otherwise, some may be complex.

The central conjecture, supported by numerical evidence, is that the statistical distribution of these zeros follows the Gaussian Unitary Ensemble (GUE) from random matrix theory. Rudnick and Sarnak provided the first major theoretical evidence for this by computing the $m$-level correlation functions. However, their results were conditional for degrees $n\geq 4$ on the assumption of Hypothesis H.

Our main theorem, which establishes Hypothesis H for all degrees, removes this condition. To state the results of Rudnick and Sarnak precisely, we first define the necessary functions and correlation sums. Let $m \geq 2$ denote the level of correlation. We define the test functions and correlation sums as follows.

\begin{enumerate}[label=(\roman*)]
\item Test function: A smooth test function $f: \mathbb{R}^m \to \mathbb{C}$ satisfies the following.
\begin{enumerate}[label=(\alph*)]
    \item $f(x_1,\ldots,x_m)$ is symmetric in its variables.
    \item $f(x+t(1,...,1)) = f(x)$ for all $t \in \mathbb{R}$.
    \item $f(x)$ decays rapidly as $|x|\to \infty$ in the hyperplane $\sum_{j=1}^m x_{j}=0$.
\end{enumerate}
\item Cutoff function: A cutoff function $h: \mathbb{C} \to \mathbb{C}$ is an entire function, constructed as the Fourier transform of a smooth, compactly supported function $g \in C_c^{\infty}(\R)$.

\item Correlation sums: Using these functions, two types of $m$-level correlation sums can be defined.
\begin{enumerate}[label=(\alph*)]
\item The smoothed $m$-level correlation sum, which does not require GRH, is given by
\[
R_{m}(T,f,h)=\sideset{}{^{\prime}}{\sum}_{j_{1},...,j_{m}} h\left(\frac{\gamma_{\pi, j_{1}}}{T}\right)\cdots h\left(\frac{\gamma_{\pi, j_{m}}}{T}\right)f\left(\frac{L}{2\pi}\gamma_{\pi, j_{1}},...,\frac{L}{2\pi}\gamma_{\pi, j_{m}}\right),
\]
where the sum is over distinct indices and $L = n \log T$.
\item To define the $m$-level correlation sum, we require the GRH. In this case, the imaginary parts $\gamma_{\pi,j}$ are real and ordered by value: $0 < \gamma_{\pi,1} \leq \gamma_{\pi,2} \leq \cdots$. The zeros are then normalized as $\tilde{\gamma}_{\pi,j}=(n/2\pi)\gamma_{\pi,j}\log|\gamma_{\pi,j}|$. Let $B_N = \{\tilde{\gamma}_{\pi,1}, \dots, \tilde{\gamma}_{\pi,N}\}$ be this set of normalized zeros. The sum is defined as
\[
R_{m}(B_{N},f) = \frac{m!}{N}\sum_{S\subseteq B_N, |S|=m}f(S).
\]
\end{enumerate}
\end{enumerate}                     

With these objects defined, the results of Rudnick and Sarnak can be stated as follows.

\begin{theorem}\label{thm-RS-smooth}
Let $\pi \in \mathfrak{F}_n$ over $\Q$, and let $f$ and $h$ be functions as defined above. Assume that $\hat{f}$ is compactly supported in $\sum_{j=1}^{m} |\xi_j| < 2/n$ (so $f$ extends to an entire function on $\mathbb{C}^m$). Then as $T \to \infty$,
\[
R_{m}(T,f,h) \sim \frac{n}{2\pi}T \log T \int_{-\infty}^{\infty}h(r)^{m}\dd r \int_{\R^{m}}f(x)W_{m}(x)\delta\left(\frac{x_{1}+\dots+x_{m}}{m}\right)\dd x_{1}\cdots \dd x_{m},
\]
where $W_m(x)$ is the $m$-level correlation density for the GUE model, defined by $W_m(x) = \det(K(x_i-x_j))$ with $K(x) = \frac{\sin(\pi x)}{\pi x}$, and $\delta(x)$ is the Dirac mass at zero.
\end{theorem}

\begin{proof}
This is Theorem 1.1 in Rudnick and Sarnak \cite{RS-1996}, whose proof is conditional on either $n \leq 3$ or the validity of Hypothesis H. Since Theorem \ref{thm-HH} establishes Hypothesis H for all degrees $n$, the result now holds unconditionally.
\end{proof}

\begin{theorem}
With the notation in Theorem \ref{thm-RS-smooth} and also GRH for $L(s,\pi)$, as $N \to~\infty$, 
\[
R_{m}(B_{N},f) \rightarrow \int_{\R^{m}}f(x)W_{m}(x)\delta\left(\frac{x_{1}+\dots+x_{m}}{m}\right)\dd x_{1}\cdots \dd x_{m}.
\]
\end{theorem}

\begin{proof}
This is Theorem 1.2 of Rudnick and Sarnak \cite{RS-1996}. They deduce this result from their Theorem 1.1. This deduction requires assuming the GRH to handle the cutoff function.
\end{proof}

These results establish the conjectured connection between the statistics of zeros of automorphic $L$-functions and random matrix theory. Rudnick and Sarnak proved that GUE statistics are a universal feature of primitive $L$-functions. This universality is notable, as the distribution of the coefficients $a_\pi(p)$ is not universal. Our verification of Hypothesis H for all degrees $n$ thus resolves the conditional nature of this foundational result for $n\geq 4$.

\section{Properties of the na\"ive convolution}

In this section, we investigate the analytic behavior of the Dirichlet series
\[ 
L(s,\pi \otimes \pi') = \sum_{\kn} \lambda_{\pi}(\kn) \lambda_{\pi'}(\kn) \N\kn^{-s},
\]
which is referred to as the na\"ive convolution of $\pi$ and $\pi'$ \cite{DK-2000}. This series naturally arises in analytic number theory, frequently appearing as the main term in the so-called ``diagonal contribution" to problems such as moment estimates or large sieve inequalities (see e.g.~\cite{DK-2000,KR-2014}).

To analyze this series effectively, it is crucial to relate it to the more structured Rankin–Selberg convolution $L(s, \pi \times \pi')$. This connection is made by introducing a correction factor $H(s, \pi, \pi')$, defined locally for each prime ideal $\kp$ by the identity
\[
L(s,\pi_{\kp}\otimes\pi_{\kp}')=L(s,\pi_{\kp}\times\pi_{\kp}')H(s,\pi_{\kp}, \pi_{\kp}').
\]
The global factor is then given by the Euler product $H(s,\pi, \pi'):=\prod_{\kp} H(s,\pi_{\kp}, \pi_{\kp}')$. This factor allows us to express the na\"ive convolution in terms of the Rankin--Selberg $L$-function
\[
L(s,\pi\otimes\pi')=L(s,\pi\times\pi')H(s,\pi, \pi').
\]
This factorization makes $H(s, \pi, \pi')$ a powerful analytic tool, and its properties are presented in the following lemma.

\begin{lemma}\label{lem-RSdeco}
Let $\pi, \pi'\in\mathfrak{F}_n$.
The Euler product $H(s,\pi, \pi')$ converges absolutely for $\Re(s)> 1-(n^2+1)^{-1}$, and in this region we have the factorization
\[
L(s,\pi\otimes\pi')=L(s,\pi\times\pi')H(s,\pi, \pi').
\]
Moreover, the following two bounds hold.
\begin{itemize}
    \item For any $\varepsilon > 0$, we have
    \[
    H(s,\pi, \pi')\ll_{n,[F:\Q],\varepsilon}(C(\pi)C(\pi'))^{\varepsilon} \quad \text{for}\;\Re(s)\geq 1-(n^2+1)^{-1}+\varepsilon.
    \]
    \item In the special case $\pi'=\tilde{\pi}$ and $s=1$, we have the lower bound 
    \[
    H(1,\pi, \tilde{\pi})\gg_{n,[F:\Q],\varepsilon}C(\pi)^{-\varepsilon}.
    \]
\end{itemize}
\end{lemma}

\begin{remark}
The analytic properties of $H(s, \pi, \pi')$ are key to understanding the behavior of the na\"ive convolution. This factor has been studied extensively in prior work, notably by Duke and Kowalski \cite{DK-2000} and Kowalski and Ricotta \cite[Section 5]{KR-2014}. In the latter, the argument depends on the analytic continuation of $H(s,\pi,\pi')$ to the region $\Re(s)>1/2+\delta_n$. However, a direct estimation of the coefficients using only the bound of Luo, Rudnick and Sarnak is insufficient to achieve this. Indeed, a careful calculation shows the coefficient of the ${\N\kp}^{-2s}$ term in the Euler product is of size $O({\N\kp}^{4\delta_n})$, which yields a region of convergence of only $\Re s > 1/2 + 2\delta_n$. Our analysis in Lemma~\ref{lem-RSdeco}, which relies on the new input from Theorem \ref{thm-main}, is therefore essential to fill this analytic gap.
\end{remark}

\begin{remark}\label{rem-naive-mu}
An analogous statement holds when the na\"ive convolution of $\pi$ and $\pi'$ is restricted to square-free integral ideals
\[ 
L^{\flat}(s,\pi \otimes \pi') = \sum_{\kn} \mu_{F}(\kn)^2\lambda_{\pi}(\kn) \lambda_{\pi'}(\kn) \N\kn^{-s},
\]
where $\mu_F$ is the M\"obius function on integral ideals as before. By the same argument of Lemma \ref{lem-RSdeco}, we obtain the factorization 
\[
L^{\flat}(s,\pi \otimes \pi') = L(s,\pi\times\pi')G(s,\pi, \pi'),
\]
where $G(s,\pi, \pi')$ satisfies the same analytic properties as $H(s,\pi, \pi')$.
\end{remark}

\begin{proof}
We compare the na\"ive convolution $L(s,\pi\otimes\pi')$ with the Rankin--Selberg convolution $L(s,\pi\times\pi')$. The Rankin--Selberg convolution has an Euler product by the general theory, and the na\"ive convolution also has one because it is a Dirichlet series whose coefficients are multiplicative. For each prime ideal $\kp$, we have
\[
L(s,\pi_{\kp}\otimes\pi_{\kp}')= \sum_{k=0}^{\infty}  \lambda_{\pi}(\kp^k) \lambda_{\pi'}(\kp^k) \N\kp^{-ks},
\]
while the inverse of the local factor $L(s,\pi_{\kp}\times\pi_{\kp}')$ admits an expansion
\[	L(s,\pi_{\kp}\times\pi_{\kp}')^{-1}= \sum_{k=0}^{\infty}   \mu_{\pi\times\pi'}(\kp^k)\N\kp^{-ks}. 
\]
Note that $\mu_{\pi\times\pi'}(\kp^k)=0$ for $k>nn'$. Thus, the quotient
\[
H(s,\pi_{\kp}, \pi_{\kp}')=\frac{L(s,\pi_{\kp}\otimes\pi_{\kp}')}{L(s,\pi_{\kp}\times\pi_{\kp}')}= \sum_{k=0}^{\infty} b_{\pi}(k,\kp) \N\kp^{-ks},
\]
where the coefficient $b_{\pi}(k,\kp)$ is given by 
\[
b_{\pi}(k,\kp)=\sum_{\nu=0}^k \lambda_{\pi}(\kp^{\nu}) \lambda_{\pi'}(\kp^{\nu}) \mu_{\pi\times\pi'}(\kp^{k-\nu}).
\]

For the prime ideals $\kp$ with $\kp\nmid \kq_{\pi}\kq_{\pi'}$, it follows from \eqref{eq-RS-unramified} and the definitions of $\lambda_{\pi}(\kn), \mu_{\pi\times\pi'}(\kn)$ that 
\[
b_{\pi}(1,\kp)=\lambda_{\pi}(\kp) \lambda_{\pi'}(\kp)+\mu_{\pi\times\pi'}(\kp)=0.
\]
The estimation of the higher-order terms proceeds as follows. For $k\geq n^2+2$, applying the bound \eqref{eq-LRS-2}, we obtain
\[
b_{\pi}(k,\kp)\N\kp^{-k\sigma_1}\ll \N\kp^{k(2\delta_{n}-\sigma_1)}\ll \N\kp^{-\frac{k}{n^2+1}},
\]
where $\sigma_1=1-(n^2+1)^{-1}+\varepsilon$. Summing over such $k$, we deduce
\[
\sum_{k\geq n^2+2}\frac{|b_{\pi}(k,\kp)|}{\N\kp^{k\sigma_1}}\ll \sum_{k\geq n^2+2} \N\kp^{-\frac{k}{n^2+1}}\ll \N\kp^{-1-\frac{1}{n^2+1}}.
\]
Combining the estimates above, the product over the unramified prime ideals satisfies 
\begin{equation}\label{eq-H-unramified}
  \begin{aligned}
\prod_{\kp\nmid \kq_{\pi}\kq_{\pi'}}H(s,\pi_{\kp}, \pi_{\kp}')\ll&\prod_{\kp\nmid \kq_{\pi}\kq_{\pi'}}\left(1+\sum_{k=2}^{n^2+1}\frac{|b_{\pi}(k,\kp)|}{\N\kp^{k\sigma_1}}\right)\left(1+\sum_{k\geq n^2+2}\frac{|b_{\pi}(k,\kp)|}{\N\kp^{k\sigma_1}}\right)\\
\ll&\prod_{\kp\nmid \kq_{\pi}\kq_{\pi'}}\left(1+\sum_{k=2}^{n^2+1}\frac{|b_{\pi}(k,\kp)|}{\N\kp^{k\sigma_1}}\right).
\end{aligned}
\end{equation}
  
Next, we treat the terms with $2\leq k\leq n^2+1$. By the inequalities in \eqref{ineq-2rs} and Lemma \ref{lem-ineq}, we have 
\[
\begin{aligned}
|b_{\pi}(k,\kp)|\leq& \sum_{\nu=0}^k \big|\lambda_{\pi}(\kp^{\nu}) \lambda_{\pi'}(\kp^{\nu}) \mu_{\pi\times\pi'}(\kp^{k-\nu})\big|\\
\leq& \frac{1}{2}\sum_{\nu=0}^k \left(\lambda_{\pi\times\tilde{\pi}}(\kp^{\nu})\lambda_{\pi\times{\tilde{\pi}}}(\kp^{k-\nu})+\lambda_{\pi'\times\tilde{\pi}'}(\kp^{\nu})\lambda_{\pi'\times{\tilde{\pi}'}}(\kp^{k-\nu}) \right)\\
\leq&  \lambda_{\pi\times{\tilde{\pi}}}(\kp^{k})+\lambda_{\pi'\times{\tilde{\pi}'}}(\kp^{k})+\frac{1}{2}\sum_{\nu=1}^{k-1} \left(\lambda_{\pi\times\tilde{\pi}}(\kp^{\nu})+\lambda_{\pi'\times\tilde{\pi}'}(\kp^{\nu}) \right)\N\kp^{2(k-\nu)(\delta_{n}+\varepsilon)}.
\end{aligned}
\] 
Therefore, summing over $k$ in the range $2\leq k\leq n^2+1$, we find
\[
\begin{aligned}
\sum_{k=2}^{n^2+1}\frac{|b_{\pi}(k,\kp)|}{\N\kp^{k\sigma_1}}&\leq \sum_{k=2}^{n^2+1}\frac{\lambda_{\pi\times{\tilde{\pi}}}(\kp^{k})+\lambda_{\pi'\times{\tilde{\pi}'}}(\kp^{k})}{\N\kp^{k\sigma_1}}\\ &\quad+\sum_{\nu=1}^{n^2} \frac{\lambda_{\pi\times{\tilde{\pi}}}(\kp^{\nu})+\lambda_{\pi'\times{\tilde{\pi}'}}(\kp^{\nu})}{\N\kp^{\nu\sigma_1}} \sum_{k=\nu+1}^{n^2+1}\N\kp^{(k-\nu)(2\delta_{n}+2\varepsilon-\sigma_1)}\\
&\ll \sum_{k=2}^{n^2+1}\frac{\lambda_{\pi\times{\tilde{\pi}}}(\kp^{k})+\lambda_{\pi'\times{\tilde{\pi}'}}(\kp^{k})}{\N\kp^{k\sigma_1}} +\sum_{\nu=1}^{n^2} \frac{\lambda_{\pi\times{\tilde{\pi}}}(\kp^{\nu})+\lambda_{\pi'\times{\tilde{\pi}'}}(\kp^{\nu})}{\N\kp^{\nu\sigma_1+\frac{1}{n^2+1}}}\\
&\ll \frac{\lambda_{\pi\times{\tilde{\pi}}}(\kp)+\lambda_{\pi'\times{\tilde{\pi}'}}(\kp)}{\N\kp^{1+\varepsilon}}+ \sum_{k=2}^{n^2+1}\frac{\lambda_{\pi\times{\tilde{\pi}}}(\kp^{k})+\lambda_{\pi'\times{\tilde{\pi}'}}(\kp^{k})}{\N\kp^{k\sigma_1}}.
\end{aligned}
\] 
Inserting this into \eqref{eq-H-unramified}, and using the estimate \eqref{eq-convexity-2} and Theorem \ref{thm-main}, the product over the unramified primes satisfies 
  \[
  \begin{aligned}
 \prod_{\kp\nmid \kq_{\pi}\kq_{\pi'}}H(s,\pi_{\kp}, \pi_{\kp}')\ll&\prod_{\kp}\left(1+O\left(\frac{\lambda_{\pi\times{\tilde{\pi}}}(\kp)+\lambda_{\pi'\times{\tilde{\pi}'}}(\kp)}{\N\kp^{1+\varepsilon}}+ \sum_{k=2}^{n^2+1}\frac{\lambda_{\pi\times{\tilde{\pi}}}(\kp^{k})+\lambda_{\pi'\times{\tilde{\pi}'}}(\kp^{k})}{\N\kp^{k\sigma_1}}\right)\right)\\
 \ll& (C(\pi)C(\pi'))^{\varepsilon}
 \end{aligned}
  \]
  for $\Re s\geq \sigma_1$. 
  
  Finally, we consider the contribution from the ramified prime ideals. Using the bound \eqref{eq-LRS-2}, we derive
 \[
 \prod_{\kp\mid \kq_{\pi}\kq_{\pi'}}H(s,\pi_{\kp}, \pi_{\kp}')\ll \prod_{\kp\mid \kq_{\pi}\kq_{\pi'}}\left(\sum_{k=0}^{\infty} \N\kp^{-\frac{k}{n^2+1}}\right)\ll(\N\kq_{\pi}\kq_{\pi'})^{\varepsilon}
 \]
 for $\Re s\geq \sigma_1$, which establishes the desired upper bound.

Next we turn to the proof of the lower bound for $H(1, \pi, \tilde{\pi})$. This is equivalent to proving the upper bound $H(1,\pi, \tilde{\pi})^{-1}\ll C(\pi)^{\varepsilon}$. 
The local factor of the inverse is
\[
H(s,\pi_{\kp}, \tilde{\pi}_{\kp})^{-1} = \frac{L(s,\pi_{\kp}\times\tilde{\pi}_{\kp})}{L(s,\pi_{\kp}\otimes\tilde{\pi}_{\kp})} := \sum_{k=0}^{\infty} d_{\pi}(k,\kp) \N\kp^{-ks}.
\]
Its coefficient $d_{\pi}(k,\kp)$ is the convolution
\[
d_{\pi}(k,\kp) = \sum_{\nu=0}^k \lambda_{\pi\times\tilde{\pi}}(\kp^{\nu}) \mu_{\pi\otimes\tilde{\pi}}(\kp^{k-\nu}),
\]
where the coefficients $\mu_{\pi\otimes\tilde{\pi}}(\kp^{m})$ of $L(s,\pi_{\kp}\otimes\tilde{\pi}_{\kp})^{-1}$ are defined by the recursive identity
\[
\mu_{\pi\otimes\tilde{\pi}}(\kp^{m})=-\sum_{j=1}^{m} |\lambda_{\pi}(\kp^j)|^2 \mu_{\pi\otimes\tilde{\pi}}(\kp^{m-j}), \quad \text{with}\;\mu_{\pi\otimes\tilde{\pi}}(1)=1.
\]
For unramified prime ideals $\kp \nmid \kq_{\pi}$, we have $d_{\pi}(0,\kp)=1$. Crucially, for $k=1$,
\[
d_{\pi}(1,\kp) = \lambda_{\pi\times\tilde{\pi}}(\kp) + \mu_{\pi\otimes\tilde{\pi}}(\kp) = |\lambda_{\pi}(\kp)|^2 - |\lambda_{\pi}(\kp)|^2 = 0.
\]
The vanishing of the linear term $d_{\pi}(1,\kp)$ is the key. The remainder of the argument, bounding the higher-order terms and the contribution from the ramified primes, proceeds in a manner entirely parallel to the upper bound estimate. This yields the bound
\[
H(1,\pi, \tilde{\pi})^{-1} \ll C(\pi)^{\varepsilon}.
\]
This completes the proof.
\end{proof}

\subsection{Applications: Strong multiplicity one for ${\rm GL}_n$}

The strong multiplicity one theorem quantifies a fundamental principle in the theory of automorphic forms: a cuspidal automorphic representation is uniquely determined by its local components. The analytic version of this principle asks: how far must one search among the local data to distinguish two distinct representations?

This paper focuses on a specific formulation of this question, cast directly in terms of the coefficients of the associated $L$-functions. For each cuspidal automorphic representation $\pi \in \mathfrak{F}_n$, its analytic conductor $C(\pi) \geq 1$ measures its arithmetic and spectral complexity. We define the bounded family
\[
\mathfrak{F}_n(Q) := \left\{ \pi \in \mathfrak{F}_n : C(\pi) \leq Q \right\}.
\]
The question we address is the following (see \cite[Section 5.1]{Michel-2007}).

\begin{problem}
Given two distinct representations $\pi, \pi' \in \mathfrak{F}_n(Q)$, what is the best upper bound for
\[
N(\pi, \pi') := \min\left\{ \N\mathfrak{n} : \lambda_\pi(\mathfrak{n}) \neq \lambda_{\pi'}(\mathfrak{n}) \right\}?
\]
\end{problem}

This problem is motivated by the need to quantify the minimal amount of data required to distinguish automorphic representations. Historically, several methods have been developed, particularly for $\mathrm{GL}_2$ over $\mathbb{Q}$. Goldfeld and Hoffstein~\cite{GH-1993} proved that two distinct holomorphic newforms are distinguished by their first $\ll Q^2 \log Q$ Fourier coefficients, unconditionally. Moreover, as described by Michel \cite[Section 5.1]{Michel-2007}, other approaches for $\mathrm{GL}_2$ include a geometric method based on the Riemann--Roch theorem, and an analytic method relying on Rankin--Selberg $L$-functions. Among these, the $L$-function approach is the most flexible and amenable to generalization to higher rank.

In the general $\mathrm{GL}_n$ context, prior results have mainly addressed a related but different question: bounding the norm of the smallest prime ideal $\mathfrak{p}$ such that $\pi_\mathfrak{p} \not\cong \pi'_\mathfrak{p}$. Moreno \cite{Moreno-1985} first established the existence of a prime ideal $\mathfrak{p}$ such that $\pi_\mathfrak{p} \not\cong \pi'_\mathfrak{p}$, with an explicit upper bound of the form
\[
\N\mathfrak{p} \ll_{n,[F:\mathbb{Q}]} \exp\left(c (\log Q)^2\right),
\]
where $c$ is a constant depending only on $n,[F:\mathbb{Q}]$. This subexponential bound was later improved. As a consequence of Brumley's method \cite{Brumley-2006-ajm} (see Liu and Wang \cite{LW-2009}), one obtains the polynomial bound
\[
\N\mathfrak{p} \ll_{n,[F:\mathbb{Q}],\varepsilon} Q^{\frac{17}{2}n - 4 + \varepsilon}.
\]
Subsequently, Liu and Wang \cite{LW-2009} sharpened this further to
\[
\N\mathfrak{p} \ll_{n,[F:\mathbb{Q}],\varepsilon} Q^{2n + \varepsilon},
\]
which remains the best known bound in the local representation setting. 

In these estimates, $\pi_\mathfrak{p}$ or $\pi'_\mathfrak{p}$ may be ramified at  $\mathfrak{p}$. In that case, the condition $\pi_\mathfrak{p}\not\cong\pi'_\mathfrak{p}$ can not imply the existence of $k\geq1$ such that $\lambda_\pi(\mathfrak{p}^k)\neq\lambda_{\pi'}(\mathfrak{p}^k)$. For example, if $n\geq 2$ and $\pi_\mathfrak{p}\not\cong\pi'_\mathfrak{p}$ are supercuspidal, then $\lambda_\pi(\mathfrak{p}^k)=\lambda_{\pi'}(\mathfrak{p}^k)=0$ for all $k\geq 1$. Thus, these results do not directly yield an upper bound for the smallest integral ideal $\mathfrak{n}$ such that $\lambda_\pi(\mathfrak{n}) \neq \lambda_{\pi'}(\mathfrak{n})$. If both $\pi_\mathfrak{p}$ and $\pi'_\mathfrak{p}$ are unramified, however, $\pi_\mathfrak{p}\not\cong\pi'_\mathfrak{p}$ implies
that $\lambda_\pi(\mathfrak{p}^k)\neq\lambda_{\pi'}(\mathfrak{p}^k)$ for some $1\leq k\leq n$.

Brumley \cite[Theorem 7]{Brumley-2006-ajm} showed that there exists a prime ideal $\mathfrak{p}$, unramified for both $\pi$ and $\pi'$, such that $\pi_\mathfrak{p}\not\cong\pi'_\mathfrak{p}$. Moreover,
\[
\N\mathfrak{p}\ll_{n,[F:\mathbb{Q}]} Q^{B(n)},
\]
where $B(n)>0$ is effective. One can check from Brumley's argument that $B(n)$ can be chosen to be a polynomial in $n$. For this prime ideal, we obtain from the preceding discussion that 
\[
N(\pi,\pi')\ll_{n,[F:\mathbb{Q}]} Q^{nB(n)}.
\]

As an application of the estimates in Lemma \ref{lem-RSdeco}, we shall obtain the explicit bound for $N(\pi,\pi')$.

\begin{theorem}\label{thm-SMO}
Let $\pi, \pi' \in \mathfrak{F}_n(Q)$ be distinct cuspidal automorphic representations. Then we have
\[
N(\pi, \pi') \ll Q^{7n^3-5n^2+8n-5+\varepsilon},
\]
where the implied constant depends only on $[F:\Q],n$ and $\varepsilon$.
\end{theorem}

\begin{remark}
The method of Theorem \ref{thm-SMO} also guarantees that the first distinction between $\pi$ and $\pi'$ can be found at a prime ideal. Specifically, there exists a prime ideal $\kp$ satisfying
\[
\N\kp \ll_{n, [F:\mathbb{Q}],\varepsilon} Q^{7n^3-5n^2+8n-5+\varepsilon},\quad\text{and} \quad \lambda_\pi(\kp) \neq \lambda_{\pi'}(\kp).
\]
Indeed, this follows by first applying the argument to the smoothed sum restricted to square-free ideals
\[
\sum_{\kn} \mu_{F}(\kn)^2\lambda_\pi(\mathfrak{n})\overline{\lambda_{\pi'}(\mathfrak{n})}V\Big(\frac{\N\kn}{X}\Big),
\]
and using the analytic properties of $L^{\flat}(s, \pi \otimes \pi')$ (see Remark \ref{rem-naive-mu}). This establishes the existence of a square-free integral ideal $\kn$ such that
\[
\N\kn \ll_{n, [F:\mathbb{Q}],\varepsilon} Q^{7n^3-5n^2+8n-5+\varepsilon},\quad\text{and} \quad \lambda_\pi(\kn) \neq \lambda_{\pi'}(\kn).
\]
By the multiplicativity of these coefficients, such an ideal must have at least one prime factor $\kp$ for which $\lambda_\pi(\kp) \neq \lambda_{\pi'}(\kp)$, and this $\kp$ satisfies $\N\kp \leq \N\kn$.
\end{remark}

\begin{proof}
Let $X = N(\pi, \pi')/3$ and consider the smooth sum
\[
\Sigma_{V}(\pi,\pi',X) = \sum_{\kn}\lambda_\pi(\mathfrak{n})\overline{\lambda_{\pi'}(\mathfrak{n})}V\Big(\frac{\N\kn}{X}\Big),
\]
where $V$ is a smooth function compactly supported in $[1/2, 5/2]$ as defined in Section 5.1. Due to the support of $x \mapsto V(x/X)$, the sum is over integral ideals with $\N\kn \leq 5X/2< N(\pi, \pi')$. Consequently, we have $\lambda_\pi(\mathfrak{n}) = \lambda_{\pi'}(\mathfrak{n})$ for all $\N\kn \leq 5X/2$.

We first evaluate $\Sigma_{V}(\pi,\pi',X)$ to obtain an upper bound. We split into two cases for the distinct representations. First, suppose that $\pi \cong \pi' \otimes |\cdot|^{it}$ for some $t \in \mathbb{R} \setminus \{0\}$. In this case, their coefficients are related by $\lambda_\pi(\mathfrak{n}) = \lambda_{\pi'}(\mathfrak{n})\N\mathfrak{n}^{-it}$. Our assumption $\lambda_\pi(\kn) = \lambda_{\pi'}(\kn)$ gives the relation $\lambda_\pi(\kn)(1 - \N\kn^{-it}) = 0$ for all $\N\kn \leq 5X/2$. This implies that $\lambda_\pi(\kn)$ could be nonzero only if $\N\kn=e^{2\pi j/|t|}$ for some nonnegative $j$. Put $q_t=e^{2\pi / |t|}$. If $q_t^{j}$ is not an integer for any $j\in \mathbb N$, then only $\N\kn=1$ can occur. Otherwise, let $j_0$ be the least positive integer such that $q_t^{j_0}\in\mathbb N$. Clearly $q_t^{j_0}\geq 2$. The minimality of $j_0$ implies that every integer $\N\kn> 1$ satisfying $\N\kn^{-it}=1$ is of the form $q_t^{j}$ with $j_0\mid j$. Inserting the estimate \eqref{eq-LRS-finite}, we bound this sum by
\begin{equation}\label{eq:uppercase1}
\begin{aligned}
\Sigma_{V}(\pi,\pi',X) \ll& X^{1-\frac{2}{n^2+1}+\frac{\varepsilon}{2}}\sum_{\substack{X/2\leq \N\kn \leq 5X/2\\\lambda_\pi(\kn) \neq 0}}1 \ll X^{1-\frac{2}{n^2+1}+\frac{\varepsilon}{2}}\sum_{\substack{X/2\leq q_t^{j} \leq 5X/2\\ j_0|j}}1 \ll X^{1-\frac{2}{n^2+1}+\varepsilon} \\
\ll& X^{\sigma_1},
\end{aligned}
\end{equation}
where $\sigma_1 = 1-(n^2+1)^{-1}+\varepsilon$.

Next, we assume that $\pi \not\cong \pi' \otimes |\cdot|^{it}$ for all $t \in \mathbb{R}$. In this case, the $L$-function $L(s,\pi\times\tilde{\pi}')$ is entire. By Mellin's inverse transform, we can write
\[
 \Sigma_{V}(\pi,\pi',X) = \frac{1}{2\pi i}\int_{(2)}L(s,\pi\otimes\tilde{\pi}')X^{s}\hat{V}(s)\dd s.
\]   
Shifting the line of integration to $\sigma_1$, it follows from Cauchy's theorem that
\[
\begin{aligned} 
\Sigma_{V}(\pi,\pi',X)&\ll X^{\sigma_1}\int_{-\infty}^\infty \big|L(\sigma_1+it,\pi\otimes\tilde{\pi}')\hat{V}(\sigma_1+it)\big|\dd t \\ 
&= X^{\sigma_1}\int_{-\infty}^\infty \big|L(\sigma_1+it,\pi\times\tilde{\pi}')H(\sigma_1+it,\pi,\tilde{\pi}')\hat{V}(\sigma_1+it)\big|\dd t \\ &\ll X^{\sigma_1} \sup_{t} |H(\sigma_1+it,\pi,\tilde{\pi}')| \int_{-\infty}^\infty \big|L(\sigma_1+it,\pi\times\tilde{\pi}')\hat{V}(\sigma_1+it)\big|\dd t.
\end{aligned} 
\]
We apply the bounds from Lemma \ref{lem-RSdeco} for $H(s,\pi,\tilde{\pi}')$ and Lemma \ref{lem-convexity} for $L(s,\pi\times\tilde{\pi}')$. This yields the upper bound
\[
\Sigma_{V}(\pi,\pi',X) \ll Q^{\varepsilon} C(\pi\times\tilde{\pi}')^{\frac{1-\sigma_{1}}{2}}X^{\sigma_1}.
\]
Note that this upper bound strictly dominates the bound $X^{\sigma_1}$ from \eqref{eq:uppercase1}. Hence, we obtain the unified upper bound 
\begin{equation}\label{eq:unified-upperbound}
\Sigma_{V}(\pi,\pi',X) \ll Q^{\varepsilon} C(\pi\times\tilde{\pi}')^{\frac{1-\sigma_{1}}{2}}X^{\sigma_1}.    
\end{equation}

On the other hand, since $\lambda_\pi(\kn) = \lambda_{\pi'}(\kn)$ holds in the support of $x \mapsto V(x/X)$, our initial sum becomes
\[
\Sigma_{V}(\pi,\pi',X) = \Sigma_{V}(\pi,\pi,X) = \sum_{\kn}|\lambda_\pi(\mathfrak{n})|^2V\Big(\frac{\N\kn}{X}\Big).
\]
For this sum, the corresponding Rankin--Selberg $L$-function $L(s,\pi\times\tilde{\pi})$ has a simple pole at $s=1$. Applying Mellin's inversion to $L(s,\pi\otimes\tilde{\pi})$ and shifting the contour to $\sigma_1$ as before, we now pick up a residue yielding a
main term and get
\begin{equation}\label{eq:lowerbound}
\Sigma_{V}(\pi,\pi,X) = X\hat{V}(1)H(1,\pi,\tilde{\pi}) \Res_{s=1}L(s,\pi\times\tilde{\pi}) +O\big( Q^{\varepsilon} C(\pi\times\tilde{\pi})^{\frac{1-\sigma_1}{2}}X^{\sigma_1}\big). 
\end{equation}
 
Combining these estimates \eqref{eq:unified-upperbound} and \eqref{eq:lowerbound}, and using the bound $C(\pi\times\tilde{\pi}') \ll Q^{2n}$ from \eqref{eq-BH}, we get
\[
X\hat{V}(1)H(1,\pi,\tilde{\pi}) \Res_{s=1}L(s,\pi\times\tilde{\pi}) \ll Q^{n(1-\sigma_1)+\varepsilon}X^{\sigma_1}.
\]
In \cite[Appendix]{Lapid-2013}, Brumley proves the lower bound 
\[
\Res_{s=1}L(s,\pi\times\tilde{\pi}) \gg C(\pi)^{-7n+5-\varepsilon} \gg Q^{-7n+5-\varepsilon}.
\]
Together with the lower bound for $H(1,\pi,\tilde{\pi})$ from Lemma \ref{lem-RSdeco}, we get
\[
X Q^{-7n+5-\varepsilon} \ll Q^{n(1-\sigma_1)+\varepsilon}X^{\sigma_1},
\]
which yields
\[
X^{1-\sigma_1} \ll Q^{n(1-\sigma_1)+(7n-5)+\varepsilon} \implies X \ll Q^{n+(7n-5)(1-\sigma_1)^{-1}+\varepsilon} \ll Q^{7n^3-5n^2+8n-5+\varepsilon}.
\]
Recalling that $X = N(\pi, \pi')/3$, we finally obtain $N(\pi, \pi') \ll Q^{7n^3-5n^2+8n-5+\varepsilon}$. This completes the proof.
\end{proof}

\section{Prime number theorem for Rankin--Selberg $L$-functions}

\subsection{Correspondence from prime powers to primes}

A central goal of analytic number theory is to understand the distribution of prime numbers. The theory of $L$-functions, via the explicit formula, naturally connects the distribution of their zeros to summatory functions associated with the von Mangoldt function $\Lambda_F(\kn)$, i.e., weighted sums over all prime ideal powers. From an arithmetic perspective, however, the ultimate object of interest is the corresponding sum taken exclusively over prime ideals.

The transition from a sum over prime powers to a sum over primes only is a standard procedure in analytic number theory. The underlying principle is illustrated in the work of Iwaniec and Kowalski (see \cite[Exercise 6]{IK-2004}), where it is shown that the contribution of higher prime powers ($k\geq 2$) can be effectively controlled. For instance, some classical approaches might rely on a strong hypothesis, such as an $l^2$-bound of the form
\[
\sum_{\N\kn\leq x} |a_{\pi\times \tilde{\pi}}(\kn)|^2\Lambda_F(\kn)\ll x\big(xC(\pi)\big)^{\varepsilon}.
\]
Under such an assumption, one could establish the relation
\[
\sum_{\N\kp\leq x} a_{\pi\times \tilde{\pi}}(\kp)\log \N\kp=\sum_{\N\kn\leq x} a_{\pi\times \tilde{\pi}}(\kn)\Lambda_F(\kn)+O\Big(x^{\frac{3}{4}}\big(xC(\pi)\big)^{\varepsilon}\Big).
\]
This establishes a bridge between sums over primes and sums involving the von Mangoldt function. Fortunately, the strong assumption mentioned above is not required for our analysis. We can establish this connection directly, as shown in the following theorem.

\begin{theorem}\label{thm-powercontri}
Let $\pi\in \mathfrak{F}_n$ and $\pi'\in \mathfrak{F}_{n'}$. For any $\varepsilon>0$ and any $0<y\leq x$, we have
\[
\sum_{x-y<\N\kp\leq x} a_{\pi\times \pi'}(\kp)\log \N\kp=\sum_{x-y<\N\kn\leq x} a_{\pi\times \pi'}(\kn)\Lambda_F(\kn)+O\Big(x^{1-\frac{1}{\max\{n,n'\}^2+1}+\varepsilon}\max\{C(\pi),C(\pi')\}^{\varepsilon}\Big),
\]
where the implied constant depends on $n,n',[F:\Q]$ and $\varepsilon$.
\end{theorem}
\begin{proof}
The difference between the two sums in the statement is the contribution from higher prime ideal powers ($k \geq 2$)
\[
\Delta=\sum_{k=2}^\infty \sum_{x-y<\N\kp^k\leq x} a_{\pi \times \pi'}(\kp^k)\log \N\kp.
\]
The sum over the short interval is bounded by the sum up to $x$. Using the coefficient inequality \eqref{a-ineq}, we get
\[
|\Delta|\leq \sum_{k=2}^\infty \sum_{\N\kp^k\leq x} \Big(a_{\pi \times \tilde{\pi}}(\kp^k)+a_{\pi' \times \tilde{\pi}'}(\kp^k)\Big)\log \N\kp.
\]
We bound the sum involving $\pi$, as the other part is treated analogously. We may apply Rankin's trick. Let $\sigma_1=1-\frac{1}{n^2+1}+\varepsilon$. We bound this sum as follows:
\[
\begin{aligned}
\sum_{k=2}^\infty \sum_{\N\kp^k\leq x} a_{\pi \times \tilde{\pi}}(\kp^k)\log \N\kp &\ll \sum_{\kp}\sum_{k=2}^\infty  \Big(\frac{x}{\N\kp^{k}}\Big)^{\sigma_{1}}a_{\pi \times \tilde{\pi}}(\kp^k)\log\N\kp \\
&\ll x^{\sigma_{1}} \sum_{\kp}\sum_{k=2}^\infty  \frac{a_{\pi \times \tilde{\pi}}(\kp^k)}{\N\kp^{k(\sigma_{1}-\frac{\varepsilon}{2})}}\\
&\ll x^{\sigma_{1}}\prod_{\kp}\left(1+\sum_{k=2}^{\infty} \frac{a_{\pi\times\tilde{\pi}}(\kp^k)}{\N\kp^{k(\sigma_{1}-\frac{\varepsilon}{2})}}\right)\\
&\ll x^{\sigma_{1}}C(\pi)^{\varepsilon},
\end{aligned}
\]
where the last step uses Theorem \ref{thm-main}. A similar bound holds for the sum involving $a_{\pi' \times \tilde{\pi}'}(\kp^k)$. Combining these estimates completes the proof.
\end{proof}

\subsection{Prime number theorems}

The prime number theorem for Rankin--Selberg $L$-functions $L(s, \pi \times \pi')$ seeks an asymptotic formula for the sum of its coefficients over prime ideals. Historically, progress on this problem has faced two main difficulties. The first is the need to assume that one of the automorphic representations is self-contragredient. The second is the difficulty of separating the contribution of primes from that of higher prime powers in sums weighted by the von Mangoldt function.

Previous results over the rational field $\mathbb{Q}$ have addressed these issues separately. For sums over prime powers, weighted by the von Mangoldt function, Liu and Ye \cite{LY-2007} established a prime number theorem for fixed representations with a strong error term of the form $O\left(x \exp\left(-b_{1}\sqrt{\log x}\right)\right)$. However, they need to assume one of the representations was self-contragredient. More recently, the highly uniform results of Kaneko and Thorner \cite{KT-2025} achieved a similar strong error term, with a range of validity polynomial in the conductor, but only under the condition $\pi' \in \{\tilde{\pi}, \tilde{\pi}'\}$. Otherwise, they established an upper bound that is uniform and non-trivial, but very weak. An asymptotic formula for sums over primes was first obtained by Wu and Ye \cite{WY-2007}. They successfully removed the prime power contributions, providing an unconditional result for low-rank cases ($n,n' \leq 4$). For general ranks, their result was conditional on Hypothesis H and yielded a weaker error term of $O(x/\log x)$.

 Our main purpose here is to prove the prime number theorem for Rankin–Selberg $L$-functions over primes in full generality, and to generalize the setting to arbitrary number fields. The key to our approach is Theorem \ref{thm-powercontri}. Our results include uniform estimates and even stronger results for the case of fixed representations.

\subsubsection*{\textbf{Uniform estimates}} We now present effective versions of the Prime Number Theorem for Rankin--Selberg $L$-functions. Let $\mathfrak{F}_{n}^{\flat}\subset \mathfrak{F}_n$ be the set of cuspidal automorphic representations $\pi$ of $\mathrm{GL}_n(\mathbb{A}_{F})$ whose central character is unitary and trivial on the diagonally embedded copy of the positive reals. For any $\pi\in \mathfrak{F}_n$ with a unitary central character, there exist unique $\pi_{0} \in \mathfrak{F}_{n}^{\flat}$ and $t_{\pi}\in \R$ such that $\pi\cong \pi_{0}\otimes |\cdot|^{it_{\pi}}$. 

Our proof requires a zero-free region for Rankin--Selberg $L$-functions. There exists a constant $c_{1}=c_{1}(n,[F:\Q])>0$ such that $L(s, \pi \times \tilde{\pi})$ is nonzero in the region
\begin{equation}\label{eq-rszerofree-1}
    \operatorname{Re}(s) \geq 1-\frac{c_{1}}{\log \left(C(\pi)(|\operatorname{Im}(s)|+3)^{n^2[F:\Q]}\right)}
\end{equation}
apart from at most one exceptional zero $\beta_1<1$. If such a zero exists, then $\beta_1$ is real and simple. This is a standard result, established in \cite[Theorem A.1]{HB-2019} and \cite[Theorem 2.1]{HT-2022}. 

\begin{theorem}\label{cor-uniformPNT}
 Let $\pi\in \mathfrak{F}_{n}^{\flat}$. For $x\gg C(\pi)^{200[F:\Q]n^3}$, we have
\[
\sum_{\N\kp\leq x} a_{\pi\times \tilde{\pi}}(\kp)\log \N\kp=x-\frac{x^{\beta_1}}{\beta_1}+O\left(x \exp\left(\frac{-c_{2}\log x}{\log (C(\pi))+\sqrt{\log x}}\right)(\log xC(\pi))^{4}\right),
\]
where $\beta_1$ is the possible exceptional zero of $L(s,\pi\times\tilde{\pi})$. The constant $c_{2}$ and the implied constant depend only on $n,[F:\Q]$.
\end{theorem}

\begin{proof}
By Theorem \ref{thm-powercontri}, the sum over primes can be replaced by the sum over prime powers weighted by the von Mangoldt function, up to a negligible error term. It suffices to estimate 
$S=\sum_{\N\kn\leq x} a_{\pi\times \tilde{\pi}}(\kn)\Lambda_F(\kn)$.
Following the standard method of \cite[Theorem 5.13]{IK-2004} and using the zero-free region \eqref{eq-rszerofree-1}, one obtains 
\begin{equation}\label{eq-apply513}
\begin{aligned}
S=&x-\frac{x^{\beta_1}}{\beta_1}+O\bigg(x\exp\Big(-\frac{c_{3}\log x}{\log C(\pi)+\log T}\Big)(\log x C(\pi))^{4}\bigg)\\
&+O\Big(\sum_{\N\kn\leq x/T}a_{\pi\times\tilde{\pi}}(\kn)\Lambda_F(\kn)\Big)+O\Big(\sum_{x<\N\kn\leq x+x/T}a_{\pi\times\tilde{\pi}}(\kn)\Lambda_F(\kn)\Big),
\end{aligned}
\end{equation}
where $c_{3}=c_{1}/3$, and $1\leq T\leq \exp\big(\frac{1}{2}\sqrt{\log x}\big)$ is a parameter to be chosen.

By \cite[Proposition 4.1]{HT-2022}, one has 
\begin{equation}\label{eq-HT-short}
    \sum_{u<\N\kn\leq u+h}a_{\pi\times\tilde{\pi}}(\kn)\Lambda_F(\kn)\ll h
\end{equation}
for any $u\gg C(\pi)^{64[F:\Q]n^3}$ and $u^{1-\frac{1}{16n^2[F:\Q]}}\ll h \leq u$. Thus, one has
\[
\sum_{x<\N\kn\leq x+x/T}a_{\pi\times\tilde{\pi}}(\kn)\Lambda_F(\kn)\ll \frac{x}{T}
\]
for any $x\gg C(\pi)^{100[F:\Q]n^3}$. Applying a dyadic decomposition and the estimate \eqref{eq-HT-short}, one also has
\[
\sum_{\N\kn\leq x/T}a_{\pi\times\tilde{\pi}}(\kn)\Lambda_F(\kn)=\sum_{\sqrt{x/T}<\N\kn\leq x/T}a_{\pi\times\tilde{\pi}}(\kn)\Lambda_F(\kn)+O\Big(\big(\frac{x}{T}\big)^{\frac{1}{2}+\varepsilon}C(\pi)^{\varepsilon}\Big)\ll \frac{x}{T}
\]
for any $x\gg C(\pi)^{200[F:\Q]n^3}$. The error terms on the second line in \eqref{eq-apply513} can be bounded by $O(x/T)$ for $x\gg C(\pi)^{200[F:\Q]n^3}$. On taking $T=\exp\big(\frac{1}{2}\sqrt{\log x}\big)$, this theorem follows.
\end{proof}

\begin{theorem}
 Let $\pi\in \mathfrak{F}_{n}^{\flat}$ and $\pi'\in \mathfrak{F}_{n'}^{\flat}$ with $\pi'\ncong \tilde{\pi}$. Then for $x\geq \exp(\left(C(\pi) C\left(\pi'\right)\right)^{2(n+n')^2})$, we have 
\[
\sum_{\N\kp\leq x} a_{\pi\times \pi'}(\kp)\log \N\kp \ll \frac{x}{(\log x)^{\frac{1}{n n'}}},
\]
where the implied constant depends only on $n,n'$ and $[F:\Q]$.
\end{theorem}

\begin{proof}
By Theorem \ref{thm-powercontri}, it suffices to bound the sum over prime powers. The necessary zero-free region is established in works such as \cite[Theorem A.1]{Lapid-2013} and \cite[Proposition 5.2]{KT-2025}. For any $\varepsilon>0$, there exists a constant $c_{4}=c_{4}(n,n',[F:\Q],\varepsilon)>0$ such that $L\left(s, \pi \times \pi'\right) \neq 0$ in the region
\[
\Re(s) \geq 1-\frac{c_{4}}{\left(\left(C(\pi) C\left(\pi'\right)\right)^{n+n'}(3+|t|)^{nn'}\right)^{1-\frac{1}{n+n'}+\frac{\varepsilon}{2}}}.
\]
We proceed as in \cite[Theorem 5.13]{IK-2004} and give the bound
\begin{equation*}
\sum_{\N\kn\leq x} a_{\pi\times \pi'}(\kn)\Lambda_F(\kn)\ll x\exp\Big(-\frac{c_5\log x}{\big(\left(C(\pi) C(\pi')\right)^{n+n'}T^{nn'}\big)^{1-\frac{1}{n+n'}+\frac{\varepsilon}{2}}}\Big)\big(\log x C(\pi)C(\pi')\big)^{4}+\frac{x}{T},
\end{equation*}
where $c_5=c_4/3$, and $1\leq T\leq \exp\big(\frac{1}{2}\sqrt{\log x}\big)$ is a parameter. 

We choose $T=(\log x)^{\frac{1}{n n'}}$ and set $\varepsilon=\frac{1}{n+n'}$ to simplify the exponent. The condition $x\geq \exp(\left(C(\pi) C\left(\pi'\right)\right)^{2(n+n')^2})$ ensures that  the first term is bounded by $O(x(\log x)^{-A})$ for any $A>0$ and is thus negligible. This theorem now follows by an application of Theorem \ref{thm-powercontri}.
\end{proof}


\subsubsection*{\textbf{Fixed representations}}
We conclude with the prime number theorem for fixed representations. In this setting, the error terms take a more familiar shape. The following theorem is derived in part by specializing the uniform estimates above, and in part by applying zero-free regions tailored specifically to fixed representations.

\begin{theorem}\label{cor-PNT3}
 Fix $\pi\in \mathfrak{F}_{n},\pi'\in \mathfrak{F}_{n'}$. We have
\[
\sum_{\N\kp\leq x} a_{\pi\times \pi'}(\kp)\log \N\kp=\begin{cases}
    \displaystyle\frac{x^{1-iu}}{1-iu}+O\Big(x \exp\big(-b_{1}\sqrt{\log x}\big)\Big), & \text{if } \pi' \cong \tilde{\pi}\otimes |\cdot|^{iu} \text{ for } u\in \mathbb{R},\\
    \\
    O\left(\displaystyle\frac{x}{(\log x)^A}\right) \text{ for any } A > 0, & \text{otherwise},
\end{cases}
\]
where the constant $b_{1}$ and the implied constants depend on $\pi$ and $\pi'$.
\end{theorem}

\begin{proof}
For the first assertion, we can write $\pi\cong \pi_{0}\otimes |\cdot|^{it_{\pi}}$ and $\pi'\cong \tilde{\pi}_{0}\otimes |\cdot|^{it_{\pi'}}$ for some $\pi_{0}\in \mathfrak{F}_{n}^{\flat}$ and $t_{\pi}, t_{\pi'} \in \R$ with $u=t_{\pi}+t_{\pi'}$. For a fixed representation $\pi_{0}$, its conductor $C(\pi_0)$ is a fixed constant. 
Brumley's estimate for a possible exceptional zero $\beta_1$ gives
\[
\beta_1 < 1-\frac{c_{6}}{C(\pi_0)^{7n}}
\]
for some constant $c_6>0$. This ensures that $\beta_1$ is bounded away from $1$.
The uniform result from Theorem \ref{cor-uniformPNT} thus simplifies to
\[
\sum_{\N\kp\leq x} a_{\pi_0\times\tilde{\pi}_0}(\kp)\log \N\kp=x+O\left(x \exp(-b_{1}\sqrt{\log x})\right),
\]
where $b_{1}>0$ is a constant depending on $\pi$ and $\pi'$. By partial summation, we then get
\[
\sum_{\N\kp\leq x} a_{\pi\times \pi'}(\kp)\log \N\kp=\sum_{\N\kp\leq x} a_{\pi_0\times\tilde{\pi}_0}(\kp)(\N\kp)^{-iu}\log \N\kp =\frac{x^{1-iu}}{1-iu}+O\left(x \exp\left(-b_{1}\sqrt{\log x}\right)\right).
\]
This completes the proof of the first assertion.

For the second assertion, according to Harcos and Thorner \cite{HT-2025}, there exists an ineffective constant $b_2=b_2(\pi,\pi',\varepsilon)>0$ such that
\[
L\left(\sigma+it, \pi \times \pi'\right) \neq 0 \text{ for } \sigma\geq 1-\frac{b_2}{(1+|t|)^\varepsilon}.
\]
Applying the standard method with this zero-free region as in \cite[Theorem 5.13]{IK-2004} yields
\[
\sum_{\N\kn\leq x} a_{\pi\times \pi'}(\kn)\Lambda_F(\kn)\ll \frac{x}{(\log x)^A}
\]
for any $A>0$ (see also \cite[Theorem 2.1]{HT-2025}). Finally, the result for the sum over primes follows from Theorem \ref{thm-powercontri}.
\end{proof}

\subsection{Selberg's orthogonality conjecture}

The Selberg orthogonality conjecture, originally proposed by Selberg \cite{Selberg-1989}, was stated for primitive $L$-functions within the axiomatic framework of the Selberg class. This conjecture asserts that the coefficients of two distinct primitive $L$-functions at prime arguments exhibit an orthogonality relation. In the context of the Langlands program, these $L$-functions are widely conjectured to correspond to irreducible automorphic representations, and the conjecture relates the arithmetic independence of $L$-functions to the non-equivalence of the associated representations. The specific form of this conjecture relevant to automorphic forms states as follows.

\begin{conjecture}\label{conj-SOC}
For cuspidal automorphic representations $\pi\in\mathfrak{F}_{n}^{\flat}$ and $\pi'\in\mathfrak{F}_{n'}^{\flat}$, one has
\[
\sum_{\N\kp \leq x} \frac{a_\pi(\kp) a_{\pi'}(\kp)}{\N\kp} = \delta_{\pi, \pi'} \log \log x + O_{\pi,\pi'}(1), \quad \text{as } x \to \infty,
\]
where $\delta_{\pi, \pi'}= 1$ if $\pi' \cong \tilde{\pi}$, and $\delta_{\pi, \pi'}= 0$ otherwise.
\end{conjecture}

Significant progress has been made toward this conjecture for $\mathrm{GL}_n$ over $\Q$. Assuming Hypothesis H, Rudnick and Sarnak \cite{RS-1996} established the diagonal case $\pi' \cong \tilde{\pi}$. In a major advance on the off-diagonal case, Liu and Ye \cite{LY-2005} were the first to obtain a result with a strong error term of the form  $O\left(\exp\left(-b\sqrt{\log x}\right)\right)$ for some $b>0$, but their work was conditional on the GRC and also assumed that at least one of the representations was self-contragredient. In subsequent work, Liu, Wang, and Ye \cite{LWY-2005} succeeded in removing the GRC, replacing it with the weaker Hypothesis H; this important generalization came at the cost of the error term, which was weakened to $O(1)$. Later, this self-contragredient condition was removed by Avdispahi\'c and Smajlovi\'c \cite{AS-2010}. Their proof was still conditional on Hypothesis H, and yielded an error term of only $O(1)$. 

Now Hypothesis H has been established unconditionally for general number fields $F$, we prove Selberg's orthogonality conjecture in full generality for automorphic representations over $F$. A key contribution of our work is achieving strong error estimates for all cases over any number field. In fact, by applying partial summation to the prime number theorem in Theorem \ref{cor-PNT3}, we obtain a more precise version with stronger error estimates.

\begin{theorem}
Conjecture~\ref{conj-SOC} holds unconditionally. More precisely, for any cuspidal automorphic representations $\pi\in\mathfrak{F}_n$ and $\pi'\in\mathfrak{F}_{n'}$, we have
\[
\sum_{\mathrm{N}\mathfrak{p} \leq x}\frac{a_\pi(\mathfrak{p}) a_{\pi'}(\mathfrak{p})}{\mathrm{N}\mathfrak{p}} =
\begin{cases}
\log\log x + b_3 + O\left(\exp\left(-b_4\sqrt{\log x}\right)\right), \\
\\
-\displaystyle\int_{x}^{\infty}\frac{y^{-iu-1}}{\log y}\,\mathrm{d}y + b_5 + O\left(\exp\left(-b_6\sqrt{\log x}\right)\right), \\
\\
b_7 + O\left((\log x)^{-A}\right) \text{ for any } A > 0,
\end{cases}
\]
where the cases correspond to
\begin{enumerate}
\itemsep=0pt
\item[(1)] $\pi' \cong \tilde{\pi}$,
\item[(2)] $\pi' \cong \tilde{\pi}\otimes |\cdot|^{iu}$ for some $u\in \mathbb{R}\setminus\{0\}$,
\item[(3)] $\pi' \not\cong \tilde{\pi}\otimes |\cdot|^{iu}$ for any $u\in \mathbb{R}$.
\end{enumerate}
Here, $b_4, b_6$ are positive constants, and all constants $b_j$ depend on $\pi, \pi'$.
\end{theorem}

\begin{remark} The integral term in the second case has the following asymptotic expansion by repeated integration by parts
\[
\int_{x}^{\infty}\frac{y^{-iu-1}}{\log y} \dd y=\int_{\log x}^{\infty}\frac{e^{-iuv}}{v} \dd v=\frac{x^{-iu}}{iu\log x} \sum_{k=0}^{M} \frac{k!}{(-iu \log x)^k}+O\big((\log x)^{-M-2}\big),
\]
which holds for any fixed integer $M>0$.
\end{remark}

\vskip 5mm

\subsection{Hoheisel's prime number theorem}

Hoheisel's classical theorem \cite{Hoheisel-1930} asserts that there exists a constant $\delta>0$ such that
\[
\sum_{x < p \leq x + y} \log p = (1 + o(1)) y
\]
uniformly for $x^{1 - \delta} \leq y \leq x$. The proof crucially relies on two ingredients. One is the zero-free region for the Riemann zeta function $\zeta(s)$ of the form
\begin{equation}\label{eq-zerofree-L}
    \sigma > 1 - c \frac{\log \log (|t| + 3)}{\log (|t| + 3)},
\end{equation}
originally due to Littlewood, and the other is the zero density estimate
\[
N(\alpha, T) := \#\left\{ \rho = \beta + i\gamma : \zeta(\rho) = 0,\ \beta > \alpha,\ |\gamma| \leq T \right\} \ll T^{4\alpha(1 - \alpha)}(\log T)^{13},
\]
established by Hoheisel himself, where $c > 0$ is an absolute and effective constant.

When considering the automorphic analogues, such as the standard $L$-function $L(s,\pi)$ for $\pi \in \mathfrak{F}_n$, or its Rankin--Selberg convolution $L(s,\pi \times \tilde{\pi})$, a major obstacle arises: no analogue of the zero-free region~\eqref{eq-zerofree-L} is currently known for such $L$-functions. To overcome this, one must resort to log-free zero density estimates, which can partially compensate for the absence of this kind of zero-free regions.

Such a log-free zero density estimate was previously developed by Soundararajan and Thorner \cite{ST-2019}. More recently, Humphries and Thorner \cite{HT-2022} extended these ideas to the setting of number fields, obtaining a uniform log-free zero density estimate for Rankin--Selberg $L$-functions associated to $\pi$ and $\tilde{\pi}$ over an arbitrary number field. As an application, they established the following Hoheisel-type result: Let $\pi\in\mathfrak{F}_n$. Let $A>c_7$. For $x \geq C(\pi)^{c_8 A^2}$ and $x^{1 - A^{-1}} \leq h \leq x$,
\begin{equation}\label{eq-hoheisel-GLn}
\sum_{x<\N\kn\leq x+h}\Lambda_{F}(\kn)a_{\pi\times\tilde{\pi}}(\kn)=\begin{cases}
    h(1-\xi^{\beta_1-1})(1+O(e^{-c_{9} A})), &\mbox{if $\beta_1$ exists,}\\
    h(1+O(e^{-c_{9} A})), &\mbox{otherwise,}
\end{cases}
\end{equation}
where the real number $\beta_1$ is the exceptional zero, $\xi\in[x,x+h]$ satisfies $(x+h)^{\beta_1}-x^{\beta_1}=\beta_1 h \xi^{\beta_1-1}$, the implied constant is absolute, and the positive constants $c_7, c_8,c_9$ depend only on $n$ and $[F:\Q]$.

As in Hoheisel's original work, Humphries and Thorner aimed to obtain asymptotic formulas for sums over prime ideals. However, they pointed out that due to current limitations toward the GRH, one cannot rule out the possibility that the  contribution from prime powers is $\gg h$. Even under the assumption of Hypothesis H, it is insufficient to ensure that the contributions from higher prime powers are negligible.

Nevertheless, under Hypothesis \eqref{eq-hypotheis-BTZ}, which is known to hold for $n\leq 4$, Humphries and Thorner managed to obtain an asymptotic formula restricted to prime ideals. Specifically, they proved that for $\pi \in \mathfrak{F}_n$ with $n \leq 4$
\[
\sum_{x < \N\kp \leq x + h} |a_\pi(\kp)|^2 \log \N\kp =
\begin{cases}
	h(1 - \xi^{\beta_1 - 1})(1 + O(e^{-c_9 A})), & \text{if $\beta_1$ exists}, \\
	h(1 + O(e^{-c_9 A})), & \text{otherwise}.
\end{cases}
\]

In the following, we extend this result to all $\mathrm{GL}_n$ cuspidal automorphic representations.

\begin{theorem}
   Let $\pi\in\mathfrak{F}_n$. Let $A>c_7$. For $x \geq C(\pi)^{c_8 A^2}$ and $x^{1 - A^{-1}} \leq h \leq x$, we have
    \[
\sum_{x < \N\kp \leq x + h} |a_\pi(\kp)|^2 \log\N\kp =
\begin{cases}
	h(1 - \xi^{\beta_1 - 1})(1 + O(e^{-c_9 A})), & \text{if an exceptional zero $\beta_1$ exists}, \\
	h(1 + O(e^{-c_9 A})), & \text{otherwise},
\end{cases}
\]
where $\xi \in [x, x + h]$ satisfies the identity $(x + h)^{\beta_1} - x^{\beta_1} = \beta_1 h \xi^{\beta_1 - 1}$, the implied constant is absolute, and the positive constants $c_7, c_8, c_9$ depend only on $n$ and $[F:\Q]$.
\end{theorem}

\begin{proof}
    By \eqref{eq-hoheisel-GLn}, Theorem \ref{thm-powercontri} and our condition $x\geq C(\pi)^{c_{8} A^2}$, we obtain
    \[
  \sum_{x < \N\kp \leq x + h} a_{\pi \times \tilde{\pi}}(\kp) \log \N\kp =
 \begin{cases}
	h(1 - \xi^{\beta_1 - 1})(1 + O(e^{-c_9 A})), & \text{if $\beta_1$ exists}, \\
	h(1 + O(e^{-c_9 A})), & \text{otherwise}.
 \end{cases}
 \]
It remains to compare $a_{\pi \times \tilde{\pi}}(\kp)$ and $|a_\pi(\kp)|^2$. By \eqref{eq-LRS-finite} and \eqref{eq-LRS-2}, we get the difference 
    \[
  \sum_{x < \N\kp \leq x + h} (a_{\pi \times \tilde{\pi}}(\kp)-|a_\pi(\kp)|^2) \log \N\kp \ll \sum_{\kp|\kq_{\pi}} n^2 \N\kp^{1-\frac{1}{n^2+1}} \log \N\kp \ll \N\kq_{\pi}^{1-\frac{1}{2(n^2+1)}},
  \]
which is also negligible in view of the assumed lower bound on $x$. This completes the proof.
\end{proof}

\vskip 5mm
\textbf{Acknowledgements.} The author is grateful to Professors Jianya Liu and Guangshi L\"u for their encouragement and support. He would also like to thank Ziyang Tang, Jesse Thorner, Hongbo Yin, and Zihao Wang for helpful discussions. Finally, the author thanks the anonymous referee for valuable comments and suggestions that significantly improved the paper. 
\vskip 2mm
\textbf{AI usage.} GPT-6 Sol was used only for final proofreading.

\end{document}